\documentclass[10pt,reqno]{amsart}
\usepackage{graphicx}
\usepackage{indentfirst,csquotes}

\usepackage[all]{xy}
\usepackage{amssymb,amsthm,amsmath,amsfonts,mathrsfs,,amsaddr}
\usepackage{xcolor,paralist,titlesec,fancyhdr,etoolbox,enumerate,enumitem,graphicx,url,hyperref}
\usepackage[para,online,flushleft]{threeparttable}
\usepackage{float,caption,setspace,tikz,comment}
\usetikzlibrary{arrows}
\usetikzlibrary{decorations.markings}

\DeclareMathOperator*{\esssup}{ess\,sup}
\newcommand\norm[1]{\left\lVert#1\right\rVert}	
\newcommand\japbrack[1]{\left\langle#1\right\rangle}
\newcommand\absval[1]{\left|#1\right|}

\theoremstyle{definition}
\newtheorem{definition}{Definition}[section]

\newtheorem{remark}[definition]{Remark}

\theoremstyle{plain}
\newtheorem{theorem}[definition]{Theorem}
\newtheorem{lemma}[definition]{Lemma}

\titleformat{\section}
   {\normalfont\Large\bfseries\centering}{\thesection}{1em}{}
   
   \titleformat{\subsection}
   {\normalfont\bfseries}{\thesubsection}{1em}{}

\hypersetup{ colorlinks=true, linkcolor=black, filecolor=black, urlcolor=black }

\usepackage{lipsum}

\begin{document}
\title{On a coupled system of KP-type} 
\author{Jacob B. Aguilar, Ph.D.}

\begin{abstract}
A defining characteristic of the Kadomstev-Petviashvili (KP) model equation is that the well-posedness results are subject to the restriction that at all transverse positions, the mass
$\int u \,dx = \text{constant independent of $y$}.$ In 2007, for a rather general class of equations of KP type, it was shown that the zero-mass (in $x$) constraint is satisfied at any non-zero time even if it is not satisfied at initial time zero. To remedy this ``odd'' behavior, a model modification is introduced which does not impose non-physical restrictions upon the initial data. In this article, we introduce a new modified KP system, named the Non-KP model equation. After providing a variational derivation of the Non-KP model, we analyze its Hamiltonian evolutionary structure. Furthermore, we prove linear estimates in the Bourgain spaces $X^{s,b}$ corresponding to the integral equation arising from the Duhamel formulation of system.
\end{abstract} 

\maketitle

\let\thefootnote\relax\footnotetext{\textit{Dedicated to my wife, Miok Ma Aguilar.}}
\let\thefootnote\relax\footnotetext{\textit{E-mail address}: \textbf{jacob.aguilar@saintleo.edu}}
\let\thefootnote\relax 

\bigskip

\section{Introduction and Background}\label{sec:introduction}
In this work we introduce a new coupled KP model. Due to the ``odd'' behavior of the mass of a solution to the KP Cauchy problem, a model modification is introduced which does not impose non-physical restrictions upon the initial data. Listed below is the pure initial value problem for the Non-KP model equation. 
\begin{equation}
\begin{cases}
u_{t} + u_{x} + uu_{x} - u_{xxt} + v_y - v_{xxy} = 0 \quad \text{for} \quad  (x,y,t) \in \mathbb{R}^2 \times \mathbb{R}_+,&\\
v_t - v_{xxt} + u_y +uu_y = 0, &\\
u(x,y,0) = u_0, &\\
v(x,y,0) = v_0. &\\
\end{cases}
\label{NON-KP}
\end{equation}
Nonlinear dispersive equations naturally arise in physical settings as models for the unidirectional propagation of weakly dispersive long waves with weak transverse effects. Due to the increasing demand of sea transport and off shore oil exploration, many scientist and engineers are concerned with evaluating forces applied by traveling waves on various off-shore structures \cite{MA}. These applications usually reduce to solving a partial differential equation and utilizing the obtained results to estimate the pressure field.

In the past few decades, research in nonlinear dispersive waves has been very active. In many applications of nonlinear dispersive evolution equations, it is useful to have estimates for the solution or forcing term in both the spatial and temporal variables. The dispersive smoothing effect enjoyed by these equations motivated the invention of the Bourgain spaces $X^{s,b}$, which provide a convenient way to control the size of the solution to the linear problem in terms of the initial data, or of the forcing term \cite{B2,B3,BT,BLT,ST,ACD,TBB,MST1,MST2,MING}. Additionally, classes of inequalities have emerged from connections to the Fourier restriction problem. A well-known example are the Strichartz estimates, which are inequalities designed to determine the size and decay of solutions to dispersive partial differential equations in mixed norm Lebesgue spaces \cite{strichartz1977restrictions}.

In the context of fluid dynamics, dispersion of water waves refers to the phenomenon that waves of different wavelengths travel at different phase speeds. Unlike waves traveling through a nondispersive medium, dispersive waves get deformed as they evolve. Water waves also exhibit a nonlinear effect in which the phase speed is dependent on the size of the amplitude. In many modeling scenarios there is a delicate interplay between the frequency dispersion and nonlinear effects. This balance between linear dispersive and nonlinear effects is represented through the foundational long wave scaling regime of surface waves. In accordance with the transformation theory presented in \cite{CG}, the KdV equation and three-dimensional approximations, such as the KP equation can be recovered from the Boussinesq scaling regime by choosing a suitable reference frame.

The KP model and its regularized version, the Benjamin-Bona-Mahony-Kadomtsev-Petviashvili (BBM-KP) model, arise in modeling scenarios corresponding to nonlinear dispersive waves which propagate principally along the $x$-axis with weak dispersive effects undergone in the direction parallel to the $y$-axis and normal to the main direction of propagation \cite{bona2002cauchy,kadomtsev1970stability}.

A defining characteristic of this class of models is that the well-posedness results are subject to the restriction that at all transverse positions, the mass
$$\int_{\mathbb{R}} u \,dx = \text{constant independent of $y$}.$$
In 2007, for a rather general class of equations of KP type, Molinet et al \cite{MST3} , showed that the zero-mass (in $x$) constraint is satisfied at any non zero time even if it is not satisfied at initial time zero. To remedy this odd behavior, we suggest a modification to the traditional KP model equation which does not impose non-physical restrictions upon the initial data. The formal derivation of the Non-KP model equation \ref{NON-KP} is covered in Section \ref{sec:variational}.  A point of major departure from the traditional KP model is that if $(u,v)$ is a solution to the linearized version of the Non-KP model, then the mass  $m(y,t) = \int_{\mathbb{R}} u dx$ is no longer independent of the transverse variable $y$. In fact, $m$ satisfies the wave equation. To find the dispersion relation of system \ref{NON-KP}, we insert the ansatz 
\begin{equation*}
\left(u,v\right)=\left(Ae^{i(\xi x+\mu y-\omega t)},Be^{i(\xi x+\mu y-\omega t)}\right)
\end{equation*}  
into the linearized version of the system to obtain
\begin{equation*}
\begin{cases}
-\omega A+\xi A -\xi^2 \omega A +B \mu +B \xi^2 \mu = 0,&\\
A \mu -\omega B -B \xi^2 \omega = 0. &\\
\end{cases}
\end{equation*}
This results in the following matrix equation
\begin{displaymath}
 \begin{pmatrix}
   -\omega+\xi-\xi^2\omega & \mu+\xi^2\mu \\[5pt]
   \mu & -\omega-\xi^2\omega
  \end{pmatrix} \begin{pmatrix}
  A \\[5pt]
  B 
 \end{pmatrix} =\begin{pmatrix}
  0 \\[5pt]
  0 
 \end{pmatrix}.
\end{displaymath}
If the existence of a nontrivial solution is to be guaranteed, we must have that 
$$\left(-\omega-\xi^2\omega \right)^2 - \left(\xi+\xi^3 \right) \omega - \mu^2 \left(1+\xi^2 \right)=0.$$
It follows that the dispersion relation for system \ref{NON-KP} is
\begin{equation}
\omega(\xi,\mu)=\frac{\xi \pm \sqrt{\xi^2+4\mu^2\left(1+\xi^2\right)}}{2\left(1+\xi^2\right)} \qquad \hfill \text{(Dispersion relation of Non-KP model)}
\label{dispersionrelationNONKP}
\end{equation}
In what follows we define the dependent variables
\begin{equation*}
\omega_1(\xi,\mu):=\frac{\xi + \sqrt{\xi^2+4\mu^2\left(1+\xi^2\right)}}{2\left(1+\xi^2\right)}\quad \text{and} \quad \omega_2(\xi,\mu):=\frac{\xi - \sqrt{\xi^2+4\mu^2\left(1+\xi^2\right)}}{2\left(1+\xi^2\right)}.
\end{equation*}
The dispersion relation \ref{dispersionrelationNONKP} for the Non-KP model contains a term which depends on both spatial Fourier modes and is raised to the power $\frac{1}{2}$, unlike the linearized dispersion relation for the KP equation and its regularized variants \cite{kadomtsev1970stability,daspan2007comparison,aguilar2024convergence}. 

This manuscript introduces a new system of KP-type and is organized as follows: Section \ref{sec:notate} features the mathematical framework and introduces the function spaces which are adapted to the system of interest,  Section \ref{sec:variational} provides a variational derivation of the Non-KP model and motivates the model modification which does not impose non-physical restrictions upon the initial data,
Section \ref{sec:hamiltonian} is dedicated to the Hamiltonian evolutionary structure and diagonalization of the model, and Section \ref{sec:linear} contains proofs of linear estimates in the Bourgain spaces $X^{s,b}$ corresponding to the integral equation arising from the Duhamel formulation of system. ``Appendix'' \ref{sec:appendix} consists of supplementary material covering the derivation of a recursive formula for the terms in the Taylor expansion of the Dirichlet-Neumann operator.
\section{Notation and Framework}\label{sec:notate}
\subsection{Preliminaries}
In this section, an overview of the notation and mathematical framework required for obtaining our results is introduced. For an in-depth look into the function classes utilized in this work, and many other closely related topics in the fields of functional analysis and nonlinear partial differential equations, the reader is referred to \cite{tao2006nonlinear,grafakos2008classical,evans2010partial,goldstein2017semigroups}. 

Firstly, let $\mathbb{R}_+:=\{x \in \mathbb{R}:x > 0\}$ and $\mathbb{R}_0:=\mathbb{R}_+ \cup \{0\}$ denote the sets of strictly positive and non-negative real numbers, respectively. Throughout this article, $n \in \mathbb{N}$ is the dimension of the Euclidean space $\mathbb{R}^n$. For $x \in \mathbb{R}^n$, we denote $\absval{x} := \left(x_1^2+ \cdots + x_n^2 \right)^\frac{1}{2}$ to be the magnitude of the vector $x$ and $(\cdot,\cdot):= \absval{x}^2$ be the canonical inner product on $\mathbb{R}^n\times\mathbb{R}^n$. The ordered pairs $(\xi, \mu)$ and $(\xi, \mu, \tau)$ correspond to the Fourier variables dual to $(x, y)$ and $(x, y, t)$, respectively. We use $\japbrack{x}$ to denote the Japanese bracket $\japbrack{x} := \left( 1+x^2 \right)^\frac{1}{2}$ of $x$. It is an immediate consequence that $\japbrack{x}$ is comparable to $\left|x\right|$ for sufficiently large $x$ and $1$ for small $x$. Furthermore, we adopt the following convention. 
\begin{definition}\label{def:bound}
Let $A,B \in \mathbb{R}_0$, we denote $A \vee B := \max\{A,B\}$ and $A \wedge B := \min\{A,B\}$. Furthermore, the notation $A \lesssim B$ (respectively $A \gtrsim B$) means that there exists an absolute positive constant $C$ such that $A \leq C B$  (respectively $A \geq C B$).
\end{definition}
If $X$ and $Y$ are both Banach spaces, we denote a continuous embedding between them by $\hookrightarrow$, i.e. $X \hookrightarrow Y$ means that $\norm{u}_Y \lesssim \norm{u}_X$.
\subsection{Function Spaces}
All integrals will be with respect to Lebesgue measure $\lambda$ for any complex-valued measurable function $u$. The Banach space $L^\infty(\mathbb{R})$ of all real valued measurable functions  which are essentially bounded is characterized by the following norm
\begin{equation*}
	|u|_\infty := \inf \left\lbrace C \geq 0 : \lambda(\{ x: |u|> C \})=0 \right\rbrace. 
\end{equation*}
   Extensive use will be made of the Lebesgue norms
$$\norm{u}_{L^p_x(\mathbb{R}^n \rightarrow C)} := \Big(\int_{\mathbb{R}^n}|u|^p dx \Big)^\frac{1}{p}$$
\noindent where $1\leq p \leq \infty$ for complex-valued measurable functions $u:\mathbb{R}^n \rightarrow \mathbb{C}$, with the convention that $\norm{u}_{L^\infty_x(\mathbb{R}^n \rightarrow \mathbb{C})} := \esssup_{x \in \mathbb{R}^n} \left|u \right |$. In this setting, $\mathbb{C}$ can be replaced with any Banach space $X$. For example, $L^p_x(\mathbb{R}^n \rightarrow X)$ denotes the space of all measurable functions $f:\mathbb{R}^n \rightarrow X$ endowed with the following norm.
$$\norm{u}_{L^p_x(\mathbb{R}^n \rightarrow X)} := \Big(\int_{\mathbb{R}^n}\norm{u}_X^p dx \Big)^\frac{1}{p}$$
In many cases we will abbreviate $L^p_x(\mathbb{R}^n \rightarrow X)$ as $L^p_x(\mathbb{R}^n)$, or even $L^p$ provided that the context is clear. This abuse of notation will also be adopted for other function classes utilized throughout this work. Define the time localized function space $C^{t_0}_{t_1}X := C\left([t_0,t_1];X\right)$ equipped with the norm
$$\norm{u}_{C^{a}_{b}X} := \sup_{t \in [a,b]} \norm{u}_X.$$
It is well known that if $X$ is a Banach space, then so is  $C^{a}_{b}X$ \cite{OO}. A function $u:\mathbb{R}^n \rightarrow C$  is called rapidly decreasing if
$$\norm{\japbrack{x}^N u}_{L^\infty_x(\mathbb{R}^n)} < \infty$$
for all non-negative $N$. A smooth function belongs to the Schwartz class $S(\mathbb{R}^n)$ provided that all of its derivatives are rapidly decreasing. It is well know that $S(\mathbb{R}^n)$ is a Fr\'echet Space and, as a result, has a dual denoted $S'(\mathbb{R}^n)$, i.e. the space of tempered distributions \cite{OO}. For $u \in S'(\mathbb{R}^n)$ we denote by $\hat{u}$ or $\mathcal{F}(u)$ to be the space-time Fourier transform 
$$\hat{u}(\xi, \mu, \tau) := \int_{\mathbb{R}^3} e^{-i(x\xi + y\mu + t\tau)} u(x,y,t)dx dy dt.$$
\noindent Moreover, we will use $\mathcal{F}_{x,y}$ and $\mathcal{F}_t$ to denote the Fourier transform in the spatial and temporal variables, respectively. 
\begin{remark}
The pseudo-differential operator $\partial^{-1}_x$ is defined via the Fourier transform as 
$$\widehat{\partial^{-1}_x u} = \frac{1}{i\xi} \hat{u}(\xi,y).$$
Due to the singularity of the symbol $\frac{1}{\xi}$ at $\xi =0$, one requires that $\hat{u}(0,y)=0$
(the Fourier transform in the variable $x$), which is clearly equivalent to
$$\int_{\mathbb{R}}u(x,y) dx=0.$$
In what follows, $\partial^{-1}_x u\in L^2(\mathbb{R}^2)$ means there exists an $L^2(\mathbb{R}^2)$ function $g$ such 
that $u=v_x$, at least in the distributional sense. When we write $\partial^{-k}_x \partial^{m}_y$ for $\left( k,m\right) \in \mathbb{Z}^+ \times \mathbb{Z}^+$, it is implicitly assumed that the operator is well-defined. As pointed out in \cite{MST3}, this imposes a constraint on the solution $u$. This condition implies that $u$ is an $x$ derivative of a suitable parent function. This can be achieved in the following two ways:
\begin{itemize}
\item If $u\in S'(\mathbb{R}^2)$ and $\xi^{-k} \mu^{m} \hat{u}(\xi,\mu, t)\in S'(\mathbb{R}^2)$.
\item If $u(x,y,t)= \frac{\partial}{\partial_x}v(x,y,t)$ for some $v \in C^1_x(\mathbb{R})$, the space of continuous functions possessing a continuous derivative with respect to $x$. This possibility imposes a decay condition on $u$. For any fixed $y \in \mathbb{R}$ and $t \in \mathbb{R}_+$ we must have $u \rightarrow 0$ as $x \rightarrow \pm \infty$, as a result $\int_{\mathbb{R}}u(x,y,t) dx=0 \quad \text{for} \quad (y,t) \in \mathbb{R} \times \mathbb{R}_+$.
\end{itemize}
\end{remark}
Let $H^s(\mathbb{R}^2)$ denote the classical Sobolev space $H^s(\mathbb{R}^2)$ equipped with the norm
$$||u||_{s} = \Big(\int_{\mathbb{R}^2}(1 + \mu^2 + \xi^2)^s|\hat{u}(\xi,\mu)|^2 d\xi d \mu \Big)^\frac{1}{2}.$$

Define the follow class of functions
$$\mathcal{X}^s(\mathbb{R}^2):=\Big\{u : u \in H^s(\mathbb{R}^2) \cap H^{s-1}(\mathbb{R}^2)\Big\}.$$
 If $v(x,y,t)= {\partial^{-1}_x}u(x,y,t)$ and $w(x,y,t) = \int^{x'}_{-\infty}u(x',y,t) dx'$, then
$\hat{u}(t)= \hat{w}(t)$ in  $S'(\mathbb{R}^2)$ for all $t \in \left[-T,T\right]$,
thus $v=w$ due to that fact that the Fourier transform is an isomorphism on $S'(\mathbb{R}^2)$. Since $u \in C\left([-T,T]; \mathcal{X}^s(\mathbb{R}^2)\right)$ and $s>2$, it is a direct consequence that $u \in L^1_{x,y}$ \cite{M2}. Moreover, $\partial^{-1}_x u \in C\left([-T,T]; H^s(\mathbb{R}^2)\right)$ and since $S(\mathbb{R}^n)$ is dense in $H^s(\mathbb{R}^2)$, we must have that $v \rightarrow 0$ as $x \rightarrow \pm \infty$. Thus,
\begin{equation*}
\int_{\mathbb{R}}u(x',y,t) dx' = \lim_{x \rightarrow \infty}\int^{x'}_{-\infty}u(x',y,t) dx'
 :=\lim_{x \rightarrow \infty} w(x,y,t)\\
 = \lim_{x \rightarrow \infty} v(x,y,t)
 = \lim_{x \rightarrow \infty} {\partial^{-1}_x}u(x,y,t)
 = 0.
\end{equation*}
Extensive use will be made of the Bourgain spaces $X^{s,b}$, also referred to as the Dispersive Sobolev Spaces. 
\begin{definition}[Bourgain spaces]
 Let $s,b$ be in $\mathbb{R}$, $(\xi,\tau)\in \mathbb{R}^d\times\mathbb{R}$ and $h\in C(\mathbb{R}^d)$.  The space $X_{\tau=h(\xi)}^{s,b}$ is defined to be the closure of the Schwartz functions $S_{x,t}(\mathbb{R}^d\times\mathbb{R})$ under the norm
\begin{displaymath}
 \| u\|_{X_{\tau=h(\xi)}^{s,b}}:= \|\langle \xi\rangle^s\langle\tau-h(\xi)\rangle^b\widehat u(\xi,\tau)\|_{L^2_{\xi,\tau}}.
\end{displaymath}
\end{definition}
These spaces are well suited for analyzing nonlinear dispersive evolution equations only after localizing in time. Fortunately, the space $X_{\tau=h(\xi)}^{s,b}$ is well-behaved with respect to time localization and enjoys the Sobolev embedding $X^{s,b}_{\tau=h(\xi)} \hookrightarrow  C^{0}_{t}{H}^s$. If the parameter $b=0$, then the dispersion does not matter and the space becomes the larger $L^2 H^s_x$ and if $h=0$, then there is no dispersion and the space is the smaller $H^b_t H^s_x$. For a detailed study of these spaces, we refer the reader to \cite{tao2006nonlinear}. For our purposes, we will need to adapt these spaces to the system of interest. As a result, we have the following definition.

\begin{definition}[Function spaces adapted to the Non-KP model \ref{NON-KP}]
\label{SobolevSpacesNONKP}
Let $(b,s) \in \mathbb{R}\times\mathbb{R}_0$ and $i=1,2$ and define the following function spaces.
\begin{align*}
\norm{u}_{H^{b}}&=\norm{\japbrack{\tau}^b \hat{u}(\tau)}_{L_{\tau}^{2}}
\\
\norm{u}_{H^{b,s}}&=\norm{\japbrack{\xi}^2\japbrack{\tau}^b  \japbrack{\absval{\xi}+\absval{\mu}}^s \hat{u}(\xi,\mu,\tau)}_{L_{\xi,\mu,\tau}^{2}} 
\\
\norm{u}_{Z^{s}}&=\norm{\japbrack{\xi}^2 \japbrack{\absval{\xi}+\absval{\mu}}^s \hat{u}(\xi,\mu)}_{L_{\xi,\mu}^{2}} 
\\
\norm{u}_{X_i^{b,s}} &=\norm{\japbrack{\xi}^2\japbrack{\tau-\omega_i(\xi,\mu)}^b  \japbrack{\absval{\xi}+\absval{\mu}}^s \hat{u}(\xi,\mu,\tau)}_{L_{\xi,\mu,\tau}^{2}}  
\end{align*}
where 
\begin{equation*}
\omega_1(\xi,\mu)=\frac{\xi + \sqrt{\xi^2+4\mu^2\left(1+\xi^2\right)}}{2\left(1+\xi^2\right)} \quad \text{and} \quad \omega_2(\xi,\mu)=\frac{\xi - \sqrt{\xi^2+4\mu^2\left(1+\xi^2\right)}}{2\left(1+\xi^2\right)}
\end{equation*}
\end{definition}
The adapted function spaces listed above will be heavily utilized in our analysis of the Non-KP model \ref{NON-KP}. For each $i=1,2$, the spaces $X_i^{b,s} \subset S'(\mathbb{R}^3)$ are Bourgain spaces adapted to the Non-KP model. Furthermore, it is clear that $H^{b,s} \subset S'(\mathbb{R}^3)$ and $Z^s \subset H_x^2$.
\subsection{Hamiltonian Evolutionary Systems}
We provide a brief review of some basics of Hamiltonian partial differential equations. For a more comprehensive coverage, we refer the interested reader to \cite{CG}. Let $H: \mathcal{D}\subset P \rightarrow \mathbb{R}$ denote the Hamiltonian function defined on a dense subdomain $\mathcal{D}$ of $P$. The gradient is taken with respect to the inner product on $P$ and $J$ is a skew-adjoint operator called the structure map which defines the Poison bracket $\{ \cdot, \cdot \}=\left(\text{grad} (\cdot), J \text{grad} (\cdot)\right)$. It is well-known that this arrangement fixes a Poison structure on $P$ \cite{OO}. If the structure map $J$ is invertible, then the Poison structure is symplectic. The water-wave problem falls into a class of Hamiltonian evolutionary systems in which $\mathcal{D}\subset P =\left(L^2(X)\right)^n$ where $X \subset \mathbb{R}^n$. The structure map $J$ is independent of $u$, the $n$ components of the vector $u$ depend on a position variable $x \in X$ and the Hamiltonian has the form $\int_{X}H$ for a Hamiltonian density function $H$. In practice, $H$ depends on spatial derivatives of $u$ and $\mathcal{D}$ is chosen so that the integral is well-defined, e.g. $\mathcal{D}=\left(S(\mathbb{R}^n)\right)^n$.
\begin{definition}[Hamiltonian Evolutionary System]
A Hamiltonian Evolutionary System is a system of partial differential equations of the form 
\begin{displaymath}
u_t = J\text{grad} H(u)
\end{displaymath}
where $u(t)$ describes a path in a Hilbert space $P$ equipped with inner product $(\cdot, \cdot)$. 
\end{definition}
The method employed to derive the nonlinear dispersive evolution equation analyzed in this work was introduced by Walter Craig et al. \cite{CG}. This transformation theory provides a straightforward procedure to represent the relationship between the traditional symplectic structure present in the water-wave problem and various other symplectic structures emerging from long-wave approximations. Hamiltonian perturbation theory first appeared in the work of Benjamin \cite{TB}. The governing idea is to approximate a Hamiltonian evolutionary system by fixing its phase space and Poison structure and replacing the Hamiltonian density function by an approximation $H_A(u)$. This results in the following approximation
\begin{displaymath}
u_t = J\text{grad} H_A(u)
\end{displaymath}
of the original Hamiltonian system. Consider a Hamiltonian evolutionary system in which $H=H\left(u, \epsilon  \right)$ for some parameter $\epsilon$. If $H$ depends smoothly on $\epsilon$ one can expand it in the following power series
\begin{displaymath}
H\left(u, \epsilon  \right)= H_0\left(u \right) + \epsilon H_1\left(u \right) +  \epsilon ^2 H_2\left(u\right) + \cdots
\end{displaymath}
This expansion is then used to construct a sequence of Hamiltonians which approximate the original system, i.e.
\begin{displaymath}
u_t = J\text{grad} \left(\sum_{i=0}^N \epsilon^i H_i(u)\right) \quad \text{for} \quad N=1,2,\ldots
\end{displaymath}
\section{Variational derivation of the Non-KP Model}
\label{sec:variational}
The notion of modeling an ideal fluid with a free surface, being acted on by gravity is a classical problem in fluid mechanics. The surface water wave problem is described by the Euler equations coupled with appropriate boundary conditions on the bottom surface; in combination with kinematic and dynamic boundary conditions on the free surface. The nonlinearities and time dependency are both a direct result of the boundary conditions on the free surface. In the presence of weak transverse effects the unknowns are: the surface elevation, the horizontal and vertical fluid velocities and the pressure. The pressure is eliminated through means of the Bernoulli equation. 

The basic conservation laws for the water wave problem are mass, momentum and total energy (Hamiltonian) conservation. For mathematical purposes we will treat a control volume element as a subspace of $\mathbb{R}^3$, i.e. the continuum hypothesis. To this end we consider a body of water of finite depth being acted on by gravity and bounded below by an undisturbed solid impermeable surface. If we ignore the effects of viscosity and assume the flow is incompressible and irrotational, then the fluid motion is modeled by the Euler equations coupled with appropriate boundary conditions on the rigid bottom and water-air interface. Simplifying assumptions are invoked which lead to various regularized model equations valid for small-amplitude long wavelength motion.

The variational derivation of the Non-KP model
\ref{NON-KP} begins with the Hamiltonian formulation of the water wave problem as in found in \cite{CG}.
$$H(\eta, \Phi)= \frac{1}{2} \int_{\mathbb{R}} \Big[ \eta^2 + \Phi G(\eta) \Phi \Big]\ dx.\label{ham}$$
 The operator $G(\eta)$ appearing in the above Hamiltonian is the Dirichlet-Neumann operator for the fluid domain $S(\eta)=\{(x,y) :-h_0<y<\eta(x,t)\}$, where $-h_0$ is the bottom profile and $\eta$ is the free surface. This linear operator takes the initial data $\Phi$ and produces the normal derivative of the solution $\phi$ of the boundary value problem 
 \begin{equation}
\begin{cases}
\Delta \phi = 0 \quad \text{for} \quad (x,y) \in S(\eta) &\\
\nabla \phi \cdot \boldsymbol{n}=0   \quad \text{for} \quad \{(x,y): y=-h_0 \} &\\
\end{cases}
\label{1.10}
\end{equation}
where $\Phi(x,t)=\phi(x,\eta(x,t),t)$ is the trace of the potential at the free surface. More explicitly, if $\phi$ solves \ref{1.10} with Dirichlet data $\Phi(x,t)=\phi(x,\eta(x,t),t)$, then
\begin{equation*}
G(\eta) \Phi(x)\ dx = (\nabla \phi \cdot \boldsymbol{n})(x)\ dS(x). 
\label{1.11}
\end{equation*}
A substantial amount of variety in water wave models is a consequence of how this positive, symmetric, bounded, analytic operator is approximated. Due to the analyticity of the operator $G(\eta)$ it can be represented as a Taylor series expansion
\begin{equation}
G(\eta) = \sum_{j \geq 0} G_j(\eta),
\label{1.12}
\end{equation}
where each term $G_j(\eta)$ is homogeneous in $\eta$ of degree $j$. The scheme for deriving the water wave models is to substitute \ref{1.12} into the above Hamiltonian and truncate terms up to a specified order obtaining an approximate Hamiltonian. A derivation for a recursive formula of the operator $G(\eta)$ is included in the Appendix \ref{sec:appendix} for the sake of completeness.

For our purposes, it will be convenient to formulate the Hamiltonian in terms of the dependent variable $u=\Phi_x$. In this spirit we define the operator $\mathcal{K}(\eta)$ by
$$G(\eta) = D \mathcal{K}(\eta) D \quad \text{for} \quad D=-i\partial_x.$$
Since $\mathcal{K}$ inherits the analyticity of $G$, it has a Taylor expansion around zero, i.e.
\begin{equation*}
\mathcal{K}(\eta) \zeta  = \sum_{j \geq 0} \mathcal{K}_j(\eta)\zeta,\quad \text{for} \quad \mathcal{K}(\eta)=D^{-1}G_j(\eta)D^{-1}.
\label{1.13}
\end{equation*}
The operator $\mathcal{K}$ enables us to write the Hamiltonian in terms of $u$, viz.
$$H(\eta, \Phi)= \frac{1}{2} \int_{\mathbb{R}} \Big[ \eta^2 + u \mathcal{K}(\eta) u \Big]\ dx.$$
The water wave problem can then be written as a Hamiltonian system using variational derivatives of $H$ and posing the Hamiltonian equations
\begin{equation}
\begin{cases}
\eta_t = -\frac{\delta H}{\delta u} &\\
u_t = - \frac{\delta H}{\delta \eta} & \\
\end{cases}
\label{1.14}
\end{equation}
Observe that the structure map associated with \ref{1.14} is symmetric. Craig and Groves \cite{CG} explain how the Hamiltonian structure can be preserved under changes of dependent or independent variables. Following their method, we let $K=\sqrt{\mathcal{K}_0}=\sqrt{\frac{\tanh D}{D}}$ and separate out the right and left-going waves by defining
\begin{equation*}r = \frac{1}{2}(L^{-1}_1 \eta + u), \quad s = \frac{1}{2}(L^{-1}_1 \eta - u).
\end{equation*}
The right and left-moving interpretation of $r$ and $s$ can be extended away from the long wave limit by a better choice for the constant coefficient self-adjoint operator $L_1$ than its long wave asymptote of unity. The exact Hamiltonian equations for the evolution of $r,s$ and $w$ are
\begin{equation*}\left[ \begin{array}{c} \frac{\partial r}{\partial T} \\ \frac{\partial s}{\partial T} \\ \frac{\partial w}{\partial T} \end{array} \right] =\frac{1}{2}L^{-1}_1\begin{bmatrix} -\frac{\partial}{\partial X} & 0 & -\frac{\partial}{\partial z} \\ 0 & -\frac{\partial}{\partial X} &
-\frac{\partial}{\partial z}\\
-\frac{\partial}{\partial z} & -\frac{\partial}{\partial z} & 0 \end{bmatrix}  \left[ \begin{array}{c}  \frac{\delta H}{\delta r} \\ \frac{\delta H}{\delta s} \\ \frac{\delta H}{\delta w}\end{array} \right].
\end{equation*}
\begin{displaymath}
H = \frac{1}{2} \iint \Big[r(L^2_1 + K -2UL_1)r +wKw +2r(L^2_1 +K)s +s(L^2_1 + K -2UL_1)s \Big] \,dXdz.
\end{displaymath}
The KP equations are obtained by replacing the dependent variable $w(X,z,T)$ in terms of $r(X,z,T)$ i.e. to neglect $s(X,z,T)$ in the exact relationship $w_x = (r-s)_z$. For exceeding long waves the operator $rKr$ is asymptotically $r +r^3$. However, any constant coefficient self-adjoint linear operator $L_2$ with a long wave asymptote of unity can be used in a change of dependent variables $r=L_2R$, $w=L_2W$ to replace the cubic contributions $r^3$ by $R^3$. This leads to approximate $rKr$ by $RK_1L^2_2R + R^3$, where the constant coefficient self-adjoint operator $K_1$ will be chosen as some long wave approximation to the exact linear operator $K(0)$. The resulting Hamiltonian only involves the single dependent variable $R$. However, it is convenient to write it in terms of both $R$ and $W$ with the four linear operators $K_1, K_2, L_1$ and $L_2$. Thus, we obtain the following Hamiltonian
\begin{equation}H = \frac{1}{2} \iint \Big[RL_2(L^2_1 + K_1 -2UL_1)L_2R + R^3 + wK_2L^2_2W \Big] \,dXdz
\label{GeneralHamiltonian}
\end{equation}
The corresponding generalization of the KP equations are 
\begin{equation*}
\begin{cases}
L_1L^2_2(R_T - UR_X) + (\frac{3}{4}R^2 + \frac{L^2_1 + K_1}{2}L^2_2R)_X + \frac{1}{2}K_2L^2L_2W, &\\
W_X - R_z = 0. &\\
\end{cases}
\label{KP general}
\end{equation*}
One obtains the usual KP equations with 
\begin{equation} K_1 = 1- \partial^2_X, \quad K_2=L_1=L_2=1.
\label{KP Operators}
\end{equation}
Another interesting selection that is related to the BBM model is
\begin{equation}K_1 = \frac{1+ \frac{1}{6} \partial^2_X}{1- \frac{1}{6} \partial^2_X}, \quad K_2=L_1 =1, \quad L_2 = \sqrt{1 -\frac{1}{6}\partial^2_X}.
\label{BBM Operators}
\end{equation}
Retaining the dependent variables $R$ and $W$, the Hamiltonian \ref{GeneralHamiltonian} gives rise to the coupled evolution equations
\begin{equation}
\begin{cases}
L_1L^2_2(R_T - UR_X) + (\frac{3}{4}R^2 + \frac{L^2_1 + K_1}{2}L^2_2R)_X + \frac{1}{2}(K_2L^2_2W)_z=0, &\\
L_1L^2_2W_T + (\frac{3}{4}R^2 + \frac{L^2_1 + K_1}{2}L^2_2R)_z =0. &\\
\end{cases}
\label{NON-KP General}
\end{equation}
These are the Non-Kadomstev-Petviashvili (Non-KP) models. The selection (\ref{KP Operators}) in (\ref{NON-KP General}) leads to a system which has the form 
\begin{equation*}
\begin{cases}
R_T - UR_X + (\frac{3}{4}R^2 + \frac{L^2_1 + K_1}{2}L^2_2R)_X + \frac{1}{2}W_z=0, &\\
W_T + (\frac{3}{4}R^2 + \frac{L^2_1 + K_1}{2}L^2_2R)_z =0, &\\
\end{cases}
\end{equation*}
while the choice of (\ref{BBM Operators}) results in the following system
\begin{equation}
\begin{cases}
(1-\frac{1}{6} \partial^2_X)(R_T - UR_X) + (\frac{3}{4}R^2 + R)_X + \frac{1}{2}(1-\frac{1}{6} \partial^2_X)W_z=0, &\\
(1-\frac{1}{6} \partial^2_X)W_T + (\frac{3}{4}R^2 + R)_z =0. &
\label{non-kp General}
\end{cases}
\end{equation}
Dropping the various constants and letting $(R,W)=(u,v)$, one observes that system \ref{non-kp General} is the  generalized version of system \ref{NON-KP} in Section \ref{sec:introduction}. A point of 
departure from the usual KP models is that if $(u,v)$ is a solution to the linearized version of model \ref{non-kp General}, then the mass  $m(y,t) = \int_{\mathbb{R}} u dx$ is no longer independent of the transverse variable $y$. Indeed, $m$ satisfies the wave equation, as mentioned in Section \ref{sec:introduction}.
$$m_{tt} = m_{zz}, \quad \text{for} \quad (z,t) \in \mathbb{R} \times \mathbb{R}_+$$

Utilizing the same scheme as in Section \ref{sec:introduction}, one can calculate the dispersion relation for the general form of the Non-KP model \ref{NON-KP General}. For a small amplitude Fourier component proportional to $e^{i(\xi x+\mu z-\omega t)}$ we have that
\begin{displaymath}
 \begin{pmatrix}
   -2\omega+\xi\left(-2U \hat{L_1}+ \hat{L^2_1} + \hat{K_1}\right) & \mu\hat{K_2} \\[5pt]
   \mu\left( \hat{L^2_1} + \hat{K_1}\right) & -2\omega
  \end{pmatrix} \begin{pmatrix}
  \hat{R} \\[5pt]
  \hat{W} 
 \end{pmatrix} =\begin{pmatrix}
  0 \\[5pt]
  0 
 \end{pmatrix}.
\end{displaymath}
For a solution to exist we must have that 
\begin{equation*}
4\omega= \xi \left(-2U \hat{L_1}+ \hat{L^2_1}+ \hat{K_1}\right) \pm \sqrt{\xi^2 \left(-2U \hat{L_1}+ \hat{L^2_1} + \hat{K_1}\right)^2 +4\mu^2  \hat{K_2}  \left( \hat{L^2_1} + \hat{K_1}\right)}
\end{equation*}
Making the choice of regularized operators \ref{BBM Operators}, a straightforward calculation shows that the above generalized dispersion relation reduces to dispersion relation \ref{dispersionrelationNONKP}, namely
\begin{align*}
4\omega&= \xi \left(\hat{L^2_1}+ \hat{K_1}\right) \pm \sqrt{\xi^2 \left(\hat{L^2_1} + \hat{K_1}\right)^2 +4\mu^2  \hat{K_2}  \left( \hat{L^2_1} + \hat{K_1}\right)}
\\
&=\xi \left(1+\widehat{\left(\frac{1+\frac{1}{6}\partial^2_x}{1-\frac{1}{6}\partial^2_x}\right)}\right) \pm \sqrt{\xi^2 \left(  1+\widehat{\left(\frac{1+\frac{1}{6}\partial^2_x}{1-\frac{1}{6}\partial^2_x}\right)}^2\right) +4\mu^2 \left( 1 + \widehat{\left(\frac{1+\frac{1}{6}\partial^2_x}{1-\frac{1}{6}\partial^2_x}\right)}\right)}
\\
&= \xi \left(\frac{2}{1+ \frac{1}{6} \xi^2}\right) \pm \sqrt{\xi^2  \left(\frac{2}{1+ \frac{1}{6} \xi^2}\right)^2 +4\mu^2   \left(\frac{2}{1+ \frac{1}{6} \xi^2}\right)}
\\
&= \frac{2\xi}{1+ \frac{1}{6} \xi^2} \pm \sqrt{\frac{4\xi^2}{\left(1+ \frac{1}{6} \xi^2\right)^2}+\frac{4\cdot 2 \mu^2 \left(1+ \frac{1}{6} \xi^2\right)}{\left(1+ \frac{1}{6} \xi^2\right)^2}}
\\
&=2\left(\frac{\xi \pm \sqrt{\xi^2  + 2\mu^2\left(1+ \frac{1}{6} \xi^2\right)} }{1+ \frac{1}{6} \xi^2}\right)
\end{align*}
Thus, we arrive at equation \ref{dispersionrelationNONKP} up to a constant, i.e.
\begin{equation*}
\omega=\frac{1}{2}\left(\frac{\xi \pm \sqrt{\xi^2 + 2\mu^2\left(1+ \frac{1}{6} \xi^2\right)} }{1+ \frac{1}{6} \xi^2}\right).
\end{equation*}
Where the choice $U=0$ was made, as a straightforward calculation shows that the Hamiltonian of \ref{NON-KP General} is the same as the Hamiltonian of the system corresponding to $U=0$.
\section{Hamiltonian evolutionary structure and Diagonalization}
\label{sec:hamiltonian}
This section is focused on the Hamiltonian structure of system \ref{NON-KP}. We write the Non-KP model in the following form
\begin{equation}
 \partial_t\boldsymbol{\eta}-A \boldsymbol{\eta}+N(\boldsymbol{\eta})=0,
 \label{sys:ham}
\end{equation}
where 
\begin{displaymath}
 \boldsymbol{\eta}=\begin{pmatrix}
  u \\
  v
 \end{pmatrix}, \text{ } A=\begin{pmatrix}
  -(1-\partial_x^2)^{-1}\partial_x & -\partial_y\\[5pt]
   -(1-\partial_x^2)^{-1}\partial_y & 0 
 \end{pmatrix} \quad \text{and} \quad N(\boldsymbol{\eta})=\begin{pmatrix}
  (1-\partial_x^2)^{-1}uu_x\\[5pt]
   (1-\partial_x^2)^{-1}uu_y
 \end{pmatrix}.
\end{displaymath}
Let the skew-adjoint operator $J$ denote the structure map
\begin{displaymath}
  J=\begin{pmatrix}
   -(1-\partial_x^2)^{-1}\partial_x & -(1-\partial_x^2)^{-1}\partial_y\\[5pt]
   -(1-\partial_x^2)^{-1}\partial_y & 0
  \end{pmatrix}.
\end{displaymath}
Since $(1-\partial^2_x)^{-1}$ is a symmetric operator and $-\partial_x, -\partial_y$ are skew symmetric, it follows that $J$ is skew symmetric with respect to the scalar product $(\cdot, \cdot)$ on $L^2(\mathbb{R}^2; \mathbb{R}^2)$, defined as
$$(\boldsymbol{\eta}, \tilde{\boldsymbol{\eta}})= \int_{\mathbb{R}^2} u\tilde{u}+ v\tilde{v}\ dx dy.$$

Thus, system \ref{NON-KP} is Hamiltonian and is equivalent to 
\begin{displaymath}
 \partial_t\boldsymbol{\eta}=J(\text{grad} H)(\boldsymbol{\eta}),
\end{displaymath}
where $H(\boldsymbol{\eta})$ is the functional given by 
\begin{displaymath}
 H(\boldsymbol{\eta})=H(u,v)=\int_{\mathbb{R}^2}\left(\frac{v_x^2}{2}+\frac{v^2}{2}+\frac{u^2}{2}+\frac{u^3}{6}\right)d xd y.
\end{displaymath}

Upon taking variational derivatives, one can observe that the above Hamiltonian gives rise to system \ref{NON-KP}, viz. 
\begin{align*}
\partial_t\boldsymbol{\eta} & = J(\text{grad} H)(\boldsymbol{\eta})
\\
& = J(\text{grad} H)(u,v)
\\
& = \begin{pmatrix}
   -(1-\partial_x^2)^{-1}\partial_x & -(1-\partial_x^2)^{-1}\partial_y\\[5pt]
   -(1-\partial_x^2)^{-1}\partial_y & 0
  \end{pmatrix}  
 \begin{pmatrix}
    u+\frac{u^2}{2} \\[5pt]
   (1-\partial_x^2)v 
  \end{pmatrix}  
\\
& = \begin{pmatrix}
   -(1-\partial_x^2)^{-1}\partial_x \left(u+\frac{u^2}{2}\right)& -(1-\partial_x^2)^{-1}(1-\partial_x^2)\partial_yv \\[5pt]
   -(1-\partial_x^2)^{-1}\partial_y\left(u+\frac{u^2}{2}\right) & 0
  \end{pmatrix}
\\
& = \begin{pmatrix}
   -(1-\partial_x^2)^{-1}\partial_x \left(u+\frac{u^2}{2}\right)& -\partial_y v \\[5pt]
   -(1-\partial_x^2)^{-1}\partial_y\left(u+\frac{u^2}{2}\right) & 0
  \end{pmatrix}
\end{align*}
Therefore, we arrive at the operator form of system \ref{NON-KP}, viz.
\begin{equation*}
\begin{cases}
u_t= -\partial_x (1-\partial^2_x)^{-1} \left(u+\frac{u^2}{2}\right)-v_y , &\\
v_t=-\partial_y (1-\partial^2_x)^{-1} \left(u+\frac{u^2}{2}\right). &\\
\end{cases}
\label{NON-KP Operator}
\end{equation*}
As $H$ is the Hamiltonian for system \ref{NON-KP}, it directly follows that $H$ is conserved by the flow, i.e.
\begin{align*}
\frac{d}{dt} H(\boldsymbol{\eta}) & = H'(\boldsymbol{\eta})\boldsymbol{\eta}_t
\\
& = \left((\text{grad} H)(\boldsymbol{\eta}), \boldsymbol{\eta}_t \right)
\\
& = \left((\text{grad} H)(\boldsymbol{\eta}), J(\text{grad} H)(\boldsymbol{\eta})\right)
\\
& =0  \qquad \text{since $J$ is skew-adjoint}.
\end{align*}

Alternatively, the Hamiltonian of system \ref{NON-KP} can be obtained via a straightforward calculation.
\begin{equation*}
u_t + v_y +Qu_x +Quu_x= 0 \quad \text{for} \quad  Q:=(1 - \partial^2_x)^{-1},
\end{equation*}
or
\begin{equation*}
uu_t + uv_y +uQu_x +uQuu_x= 0.
\end{equation*}
We integrate over $\mathbb{R}^2$ to obtain the following string of implications
\begin{align*}
& \frac{d}{dt}\int_{\mathbb{R}^2}\frac{u^2}{2}\,dxdy  -\int_{\mathbb{R}^2}u_yv \,dxdy  + \int_{\mathbb{R}^2}uQuu_x\,dxdy =0. & \\
 &\frac{d}{dt}\int_{\mathbb{R}^2} \frac{v^2}{2} \,dxdy + \frac{d}{dt}\int_{\mathbb{R}^2} \frac{v^2_x}{2} \,dxdy + \int_{\mathbb{R}^2} vuu_y \,dxdy = -\int_{\mathbb{R}^2}u_yv \,dxdy. & \\
 &\frac{d}{dt}\int_{\mathbb{R}^2}\frac{u^2}{2}\,dxdy +  \frac{d}{dt}\int_{\mathbb{R}^2} \frac{v^2}{2} \,dxdy + \frac{d}{dt}\int_{\mathbb{R}^2} \frac{v^2_x}{2} \,dxdy + \int_{\mathbb{R}^2} vuu_y \,dxdy&
\\
& + \int_{\mathbb{R}^2}uQuu_x\,dxdy =0. \\
 \end{align*}
Therefore, 
\begin{align*}
&\frac{d}{dt}\int_{\mathbb{R}^2}\frac{u^2}{2}\,dxdy  + \frac{d}{dt}\int_{\mathbb{R}^2} \frac{v^2}{2} \,dxdy + \frac{d}{dt}\int_{\mathbb{R}^2} \frac{v^2_x}{2} \,dxdy -\int_{\mathbb{R}^2} \frac{u^2}{2}v_y \,dxdy &
\\ 
&  + \int_{\mathbb{R}^2}uQuu_x\,dxdy =0. \\
&\frac{d}{dt}\int_{\mathbb{R}^2}\frac{u^2}{2}\,dxdy +  \frac{d}{dt}\int_{\mathbb{R}^2} \frac{v^2}{2} \,dxdy + \frac{d}{dt}\int_{\mathbb{R}^2} \frac{v^2_x}{2} \,dxdy &
\\
&-\int_{\mathbb{R}^2} \frac{u^2}{2}(-u_t -Qu_x -Quu_x) \,dxdy  + \int_{\mathbb{R}^2}uQuu_x\,dxdy =0. 
\\
 &\frac{d}{dt} \int_{\mathbb{R}^2}\Big[\frac{u^2}{2} +\frac{v^2}{2} + \frac{v^2_x}{2}  + \frac{u^3}{6}\Big] \,dxdy +\int_{\mathbb{R}^2} \frac{u^2}{2}Qu_x\, dxdy &
\\
& +\int_{\mathbb{R}^2} \frac{u^2}{2}Quu_x \,dxdy  + \int_{\mathbb{R}^2}uQuu_x\,dxdy =0. 
\\
 &\frac{d}{dt} \int_{\mathbb{R}^2}\Big[\frac{u^2}{2} +\frac{v^2}{2} + \frac{v^2_x}{2}  + \frac{u^3}{6}\Big] \,dxdy &
\\
& +\int_{\mathbb{R}^2} \frac{u^2}{2}Qu_x\, dxdy +\int_{\mathbb{R}^2} \frac{u^2}{2}Q\frac{\partial}{\partial_x}\frac{u^2}{2} \,dxdy -\int_{\mathbb{R}^2}u_xQ\frac{u^2}{2}\,dxdy =0. \\
&\frac{d}{dt} \int_{\mathbb{R}^2}\Big[\frac{u^2}{2} +\frac{v^2}{2} + \frac{v^2_x}{2}  + \frac{u^3}{6}\Big] \,dxdy =0.
\end{align*}
  
Before venturing into the task of obtaining linear estimates, we transform system \ref{NON-KP} into an equivalent system with a diagonal linear part. First, observe that the Fourier transform of $A$ is
\begin{displaymath}
\hat{A}(\xi,\mu) = i\begin{pmatrix}
 -\frac{\xi}{1+\xi^2} & -\mu \\[5pt]
  -\frac{\mu}{1+\xi^2} & 0
\end{pmatrix}.  
\end{displaymath}
Now we calculate the characteristic polynomial of $\hat{A}(\xi,\mu)$. Denote the identity matrix by $\mathcal{I}$, so that
\begin{align*}
\hat{A}(\xi,\mu)&= \det{\left(\hat{A}-\lambda  \mathcal{I} \right)}
\\
&=\lambda^2 +\frac{\xi i}{1+\xi^2} \lambda+ \frac{\mu^2}{1+\xi^2}
\end{align*}
The eigenvalues are 
\begin{displaymath}
\left(\lambda_1, \lambda_2  \right)=\left( -\omega_1 i, -\omega_2 i \right),
\end{displaymath}
as expected.
Next, we search for a nonzero $\boldsymbol{\eta}$, such that
\begin{displaymath}
 \left(\hat{A}-\lambda_1 \mathcal{I} \right)\boldsymbol{\eta} = \boldsymbol{0}.
\end{displaymath}
More explicitly, we need to solve the following matrix equation.
\begin{displaymath}
 \begin{pmatrix}
   \frac{\xi}{1+\xi^2}-\omega_1 & \mu\\[5pt]
   \frac{\mu}{1+\xi^2} & -\omega_1
  \end{pmatrix} \begin{pmatrix}
  u \\[5pt]
  v 
 \end{pmatrix} =\begin{pmatrix}
  0 \\[5pt]
  0 
 \end{pmatrix}.
\end{displaymath}
The solution that produces the first eigenvector, coresponding to eigenvalue $\lambda_1$ is
\begin{equation*}
u=\frac{\left( 1+\xi^2\right)\omega_1}{\mu}v
= \frac{\xi + \sqrt{\xi^2+4\mu^2\left(1+\xi^2\right)}}{2\mu}v.
\end{equation*}
Thus,
\begin{displaymath}
E_1=
  \begin{pmatrix}
  1\\[5pt]
  \frac{\mu}{\left(1+\xi^2\right)\omega_1} 
 \end{pmatrix}.
\end{displaymath}
A similar calculation for $\lambda_2$ yields
\begin{displaymath}
E_2=
  \begin{pmatrix}
 1\\[5pt]
   \frac{\mu}{\left(1+\xi^2\right)\omega_2}  
 \end{pmatrix}.
\end{displaymath}
Define $\hat{P}$ to be the matrix $\left(E_1,E_2\right)$, in the canonical basis, so that
\begin{displaymath}
\hat{P}=
  \begin{pmatrix}
  1&  1 \\[5pt]
 \frac{\mu}{\left(1+\xi^2\right)\omega_1}   & \frac{\mu}{\left(1+\xi^2\right)\omega_2}  
 \end{pmatrix} 
\end{displaymath}
and
\begin{displaymath}
\hat{P}^{-1}=
  \begin{pmatrix}
\frac{\omega_1}{\omega_1-\omega_2} & \frac{-\omega_1 \omega_2 \left(
1+\xi^2\right)}{\mu \left(\omega_1 - \omega_2  \right)}  \\[5pt]
\frac{-\omega_2}{\omega_1-\omega_2}& \frac{\omega_1 \omega_2 \left(
1+\xi^2\right)}{\mu \left(\omega_1 - \omega_2  \right)}
 \end{pmatrix}. 
\end{displaymath}
Let $\hat{D}$ denote the diagonal matrix of eigenvalues, i.e.
\begin{displaymath}
\hat{D} = \begin{pmatrix}
 \lambda_1 & 0 \\[5pt]
  0 & \lambda_2
\end{pmatrix},  
\end{displaymath} 
so that
\begin{align*}
\hat{A}& = \begin{pmatrix}
  1&  1 \\[5pt]
 \frac{\mu}{\left(1+\xi^2\right)\omega_1}   & \frac{\mu}{\left(1+\xi^2\right)\omega_2}  
 \end{pmatrix} 
   \begin{pmatrix}
 \lambda_1 & 0 \\[5pt]
  0 & \lambda_2
\end{pmatrix} 
 \begin{pmatrix}
\frac{\omega_1}{\omega_1-\omega_2} & \frac{-\omega_1 \omega_2 \left(
1+\xi^2\right)}{\mu \left(\omega_1 - \omega_2  \right)}  \\[5pt]
\frac{-\omega_2}{\omega_1-\omega_2}& \frac{\omega_1 \omega_2 \left(
1+\xi^2\right)}{\mu \left(\omega_1 - \omega_2  \right)}
 \end{pmatrix}&  
\\
& = \hat{P} \hat{D} \hat{P}^{-1}.&
\end{align*}
Placing this together, we write 
\begin{equation}
\partial_t \hat{\boldsymbol{\eta}} - \hat{A}\hat{\boldsymbol{\eta}} + \hat{N}(\hat{\boldsymbol{\eta}})=0,
\label{NON-KP_hat}
\end{equation}
and invoke the change of variables $\hat{\boldsymbol{\eta}}=\hat{P}\hat{\boldsymbol{w}}$, for
\begin{displaymath}
\boldsymbol{w} =
\begin{pmatrix}
  w_1 \\[5pt]
  w_2 
 \end{pmatrix}.
\end{displaymath}
As a result, equation \ref{NON-KP_hat} can be rewritten in the following form
\begin{equation*}
\partial_t \hat{\boldsymbol{w}} - \hat{P}^{-1}\hat{A}\hat{P}\hat{\boldsymbol{w}} +\hat{P}^{-1} \hat{N}(\hat{P}\hat{\boldsymbol{w}})=0,
\end{equation*}
or equivalently
\begin{equation*}
\partial_t \hat{\boldsymbol{w}} -  \hat{D} \hat{\boldsymbol{w}} +\hat{P}^{-1} \hat{N}(\hat{P}\hat{\boldsymbol{w}})=0.
\label{NON-KP_COV}
\end{equation*}
To calculate $\hat{P}^{-1} \hat{N}(\hat{P}\hat{\boldsymbol{w}})$, we define the polynomial $\Lambda_p(w_1,w_2)=(w_1+w_2)^2$. Thus,
\begin{align*}
\mathcal{F}_{x,y}\left[\Lambda_p(w_1,w_2) \right]
&=\left(\widehat{w_1 + w_2}\right)\ast \left(\widehat{w_1 + w_2}\right)
\\
&:=\Lambda(\hat{w}_1, \hat{w}_2) 
\end{align*}
and it follows that
\begin{align*}
\hat{P}^{-1} \hat{N}(\hat{P}\hat{\boldsymbol{w}}) & = \frac{1}{2}\begin{pmatrix}
\frac{\omega_1}{\omega_1-\omega_2} & \frac{-\omega_1 \omega_2 \left(
1+\xi^2\right)}{\mu \left(\omega_1 - \omega_2  \right)}  \\[5pt]
\frac{-\omega_2}{\omega_1-\omega_2}& \frac{\omega_1 \omega_2 \left(
1+\xi^2\right)}{\mu \left(\omega_1 - \omega_2  \right)}
 \end{pmatrix} 
 \begin{pmatrix}
    \frac{\xi i}{\left(1+\xi^2 \right)} \left(\widehat{w_1 + w_2}\right)\ast \left(\widehat{w_1 + w_2}\right)  \\[5pt]
    \frac{\mu i}{\left(1+\xi^2 \right)} \left(\widehat{w_1 + w_2}\right)\ast \left(\widehat{w_1 + w_2}\right) 
  \end{pmatrix}  
\\
& = \begin{pmatrix}
\frac{\omega_1}{\omega_1-\omega_2} & \frac{-\omega_1 \omega_2 \left(
1+\xi^2\right)}{\mu \left(\omega_1 - \omega_2  \right)}  \\[5pt]
\frac{-\omega_2}{\omega_1-\omega_2}& \frac{\omega_1 \omega_2 \left(
1+\xi^2\right)}{\mu \left(\omega_1 - \omega_2  \right)}
 \end{pmatrix} 
\begin{pmatrix}
    \frac{\xi i}{\left(1+\xi^2 \right)} \Lambda(\hat{w}_1, \hat{w}_2) \\[5pt]
    \frac{\mu i}{\left(1+\xi^2 \right)} \Lambda(\hat{w}_1, \hat{w}_2) 
  \end{pmatrix}
\\
&= \begin{pmatrix}
    \frac{\xi i}{\left(1+\xi^2 \right)}\frac{\omega_1}{\left(\omega_1 - \omega_2  \right)} - \frac{\mu i}{\left(1+\xi^2 \right)}\frac{\omega_1 \omega_2 \left(1+\xi^2\right)}{\mu \left(\omega_1 - \omega_2  \right)} 
\\[5pt]
-\frac{\xi i}{\left(1+\xi^2 \right)}\frac{\omega_2}{\left(\omega_1 - \omega_2  \right)} + \frac{\mu i}{\left(1+\xi^2 \right)}\frac{\omega_1 \omega_2 \left(1+\xi^2\right)}{\mu \left(\omega_1 - \omega_2  \right)}        
  \end{pmatrix} 
  \begin{pmatrix}
     \Lambda(\hat{w}_1, \hat{w}_2) \\[5pt]
     \Lambda(\hat{w}_1, \hat{w}_2) 
  \end{pmatrix}
\\
&= \begin{pmatrix}
    \frac{\xi i}{\left(1+\xi^2 \right)}\frac{\omega_1}{\left(\omega_1 - \omega_2  \right)} - i\frac{\omega_1 \omega_2}{\left(\omega_1 - \omega_2  \right)} 
\\[5pt]
-\frac{\xi i}{\left(1+\xi^2 \right)}\frac{\omega_2}{\left(\omega_1 - \omega_2  \right)} + i\frac{\omega_1 \omega_2}{\left(\omega_1 - \omega_2  \right)}        
  \end{pmatrix} 
  \begin{pmatrix}
     \Lambda(\hat{w}_1, \hat{w}_2) \\[5pt]
     \Lambda(\hat{w}_1, \hat{w}_2)
  \end{pmatrix}
\\
&= \begin{pmatrix}
    i\frac{\xi  \omega_1}{\sqrt{\xi^2+4\mu^2\left(1+\xi^2\right)}} +i\frac{\mu^2}{\sqrt{\xi^2+4\mu^2\left(1+\xi^2\right)}}
\\[5pt]
-i\frac{\xi \omega_2}{\sqrt{\xi^2+4\mu^2\left(1+\xi^2\right)}} - i\frac{\mu^2}{\sqrt{\xi^2+4\mu^2\left(1+\xi^2\right)}}      
  \end{pmatrix} 
  \begin{pmatrix}
     \Lambda(\hat{w}_1, \hat{w}_2) \\[5pt]
     \Lambda(\hat{w}_1, \hat{w}_2)
  \end{pmatrix}
\\
&= \begin{pmatrix}
    i\frac{\xi\omega_1+\mu^2}{\sqrt{\xi^2+4\mu^2\left(1+\xi^2\right)}}
\\[5pt]
-i\frac{\xi \omega_2+\mu^2}{\sqrt{\xi^2+4\mu^2\left(1+\xi^2\right)}}
  \end{pmatrix} 
  \begin{pmatrix}
     \Lambda(\hat{w}_1, \hat{w}_2) \\[5pt]
     \Lambda(\hat{w}_1, \hat{w}_2) 
  \end{pmatrix}.
\end{align*}
In summary, system \ref{NON-KP} can be rewritten as
\begin{equation}
\begin{cases}
 \partial_t \hat{w}_1 - i\omega_1\hat{w}_1 +\hat{M}_1(\xi,\mu)\Lambda(\hat{w}_1, \hat{w}_2) = 0,&\\
 \partial_t \hat{w}_2 - i\omega_2 \hat{w}_2 +\hat{M}_2(\xi,\mu)\Lambda(\hat{w}_1, \hat{w}_2)= 0, &\\
\hat{w}_1(\xi,\mu,0) = {\hat{w_1}}_0, &\\
\hat{w}_2(\xi,\mu,0) = {\hat{w_2}}_0. &\\
\end{cases}
\label{NON-KP (explicitly in terms of w)}
\end{equation}
Where $M_1(\partial_x,\partial_y)$ and $M_2(\partial_x,\partial_y)$ are Fourier multipliers with symbols  $i\frac{\xi\omega_1+\mu^2}{\sqrt{\xi^2+4\mu^2\left(1+\xi^2\right)}}$ and $-i\frac{\xi \omega_2+\mu^2}{\sqrt{\xi^2+4\mu^2\left(1+\xi^2\right)}}$, respectively.
An application of Duhamel's formula allows us to write \ref{NON-KP (explicitly in terms of w)} in integral form.
\begin{equation}
\begin{cases}
 w_1(t)=S_1(t){w_1}_0-\displaystyle\int_{0}^{t} S_1(t-t')M_1(\partial_x,\partial_y)\Lambda_p(w_1,w_2)  dt',&\\
 w_2(t)=S_2(t){w_2}_0-\displaystyle\int_{0}^{t} S_2(t-t')M_2(\partial_x,\partial_y)\Lambda_p(w_1,w_2)  dt' . &\\
\end{cases}
\label{NON-KP (Duhamel)}
\end{equation}
Where $\mathcal{F}_{x,y}\left[S_i(t){w_i}_0\right]=e^{\omega_it}{\hat{w_i}}_0$ denotes the solution of the free Non-KP evolution, for  $i=1,2$.
\begin{remark}
If $(w_1,w_2)$ solves \ref{NON-KP (Duhamel)} locally, then it also solves \ref{NON-KP (explicitly in terms of w)} in a distributional sense. As a result, it also solves \ref{NON-KP}, in lieu of the change of variables. 
\end{remark}
\section{Linear Estimates in Bourgain spaces}
\label{sec:linear}
The first step in obtaining the estimates involves a localization in time. To this end, let $\psi(t) \in C_c^\infty(\mathbb{R})$ be a bump function on the real line such that supp $\psi \subset [-2,2]$, and $\psi=1$ on $[-1,1]$. Let $T \in \mathbb{R}_+$ and $\psi_T(t)=\psi(\frac{t}{T})$, since supp $\psi \subset [-2,2]$, it is a direct consequence that supp $\psi_T \subset [-2T,2T]$. Define the temporally truncated integral equation
\begin{equation}
\boldsymbol{\Omega}\begin{pmatrix}
    {w}_1 \\[5pt]
    {w}_2
  \end{pmatrix}
=  \psi(t)  \begin{pmatrix}
    S_1(t){w_1}_0 \\[5pt]
    S_2(t){w_2}_0
  \end{pmatrix}
-\psi_T(t) \displaystyle\int_{0}^{t}\begin{pmatrix}
  S_1(t-t')M_1(\partial_x,\partial_y)\Lambda_p(w_1,w_2) \\[5pt]
  S_2(t-t')M_2(\partial_x,\partial_y)\Lambda_p(w_1,w_2) 
  \end{pmatrix} dt'
  \label{NON-KP (Truncated Duhamel)}
\end{equation}
Utilizing the method in \cite{ginibre1995probleme}, we proceed to obtain linear estimates for the above integral equation in the Bourgain spaces $X_i^{s,b}$, for $i=1,2$.
\begin{lemma}
For $s \in \mathbb{R}$, it follows that
$$\norm{\psi(t) S_i(t){w_i}_0}_{X_i^{b,s}} \lesssim  \norm{{w_i}_0}_{Z^s} \quad \text{for}\quad i=1,2.$$
\label{lemma:firstterm}
\end{lemma}
\begin{proof}
Let $\psi \in C_c^\infty(\mathbb{R})$, then
\begin{align*}
\norm{\psi(t) S_i(t){w_i}_0}_{X_i^{b,s}} &=\norm{\japbrack{\xi}^2\japbrack{\tau-\omega_i(\xi,\mu)}^b  \japbrack{\absval{\xi}+\absval{\mu}}^s \mathcal{F}_t \left[\psi(t) e^{\omega_it}{\hat{w_i}}_0(\xi,\mu) \right]}_{L_{\xi,\mu,\tau}^{2}} 
\\
&=\norm{\japbrack{\xi}^2\japbrack{\tau-\omega_i(\xi,\mu)}^b  \japbrack{\absval{\xi}+\absval{\mu}}^s {\hat{w_i}}_0(\xi,\mu) \mathcal{F}_t \left[\psi(t) e^{\omega_it}\right]}_{L_{\xi,\mu,\tau}^{2}}
\\
&=\norm{\japbrack{\xi}^2 \japbrack{\tau-\omega_i(\xi,\mu)}^b \japbrack{\absval{\xi}+\absval{\mu}}^s {\hat{w_i}}_0(\xi,\mu)  \hat{\psi}(\tau - \omega_i(\xi,\mu)}_{L_{\xi,\mu,\tau}^{2}}
\\
&=\norm{\japbrack{\xi}^2 \japbrack{\absval{\xi}+\absval{\mu}}^s {\hat{w_i}}_0(\xi,\mu) \norm{\japbrack{\beta}^b \hat{\psi}(\beta)}_{L_{\beta}^{2}}}_{L_{\xi,\mu}^{2}}
\\
&=\norm{{w_i}_0}_{Z^{s}} \cdot \norm{\psi}_{H^{b}} 
\\
& \leq C \norm{{w_i}_0}_{Z^{s}} 
\end{align*}
\end{proof}
For the Duhamel term of equation \ref{NON-KP (Truncated Duhamel)}, we have the following estimates.
\begin{theorem}
 Let $\epsilon \in \left(0,\frac{1}{4}\right)$, $b=\frac{1}{2}+\epsilon$, $b'=\frac{1}{2}-2\epsilon$ and define $F_i(w_1(t'),w_2(t'))= M_i(\partial_x,\partial_y)\Lambda_p(w_1,w_2)$ for $i=1,2$. If $F_i \in X_i^{-b',s}$, then the following estimate holds.
 \begin{displaymath}
  \norm{\psi\displaystyle\int_{0}^{t} S_i(t-t')F_i(w_1(t'),w_2(t'))dt'}_{X_i^{b,s}}\lesssim  \norm{F_i(w_1,w_2)}_{X_i^{-b',s}}
 \end{displaymath}
 Furthermore, for any $T \in \mathbb{R}_+$ it follows that
 \begin{displaymath}
  \norm{\psi_T\displaystyle\int_{0}^{t} S_i(t-t')F_i(w_1(t'),w_2(t'))dt'}_{X_i^{b,s}}\lesssim T^\epsilon \norm{F_i(w_1,w_2)}_{X_i^{-b',s}}
 \end{displaymath}
 \label{main lemma}
\end{theorem}
The proof of Theorem \ref{main lemma} is accomplished by the following lemmas and a trick. Firstly, we proceed by proving the following lemma.
\begin{lemma}
For $b=\frac{1}{2}+\epsilon$ and all $T \in \mathbb{R}_+$, it follows that
\begin{displaymath}
\norm{\psi_T}_{H^b} \lesssim T^{\frac{1}{2}-b} 
\end{displaymath}
 \label{lemma:psiT}
\end{lemma}
\begin{proof}
Let $\psi_T(t):=\psi(\frac{t}{T})\in C_c^\infty(\mathbb{R})$, then
\begin{align*}
\norm{\psi_T}_{H^b} &= \left(\displaystyle\int_{\mathbb{R}} \japbrack{\tau}^b T^2 {\hat{\psi}}^2(T\tau) d\tau\right)^\frac{1}{2},
\\
&= \left(\displaystyle\int_{\mathbb{R}} \left(1+\tau^2\right)^b T^2 {\hat{\psi}}^2(T\tau) d\tau\right)^\frac{1}{2}.
\end{align*}
We make the change of variables $\gamma=T\tau$, so that $d\gamma=Td\tau$, to obtain
\begin{align*}
&= \left(\displaystyle\int_{\mathbb{R}} \left(1+\frac{\gamma^2}{T^2}\right)^b T {\hat{\psi}}^2(\gamma) d\gamma\right)^\frac{1}{2}
\\
& \leq C\left(\displaystyle\int_{\mathbb{R}} \frac{\gamma^{2b}}{T^{2b}} T {\hat{\psi}}^2(\gamma) d\gamma\right)^\frac{1}{2}
\\
& = C\left(\displaystyle\int_{\mathbb{R}} \gamma^{2b} T^{-2b+1}{\hat{\psi}}^2(\gamma) d\gamma\right)^\frac{1}{2}
\\
& = CT^{\frac{1}{2}-b} \left(\displaystyle\int_{\mathbb{R}} \gamma^{2b}{\hat{\psi}}^2(\gamma) d\gamma\right)^\frac{1}{2}
\\
& \leq C T^{\frac{1}{2}-b} \quad \text{since $C_c^\infty(\mathbb{R}) \subset \mathcal{S}(\mathbb{R})$}.
\end{align*}
\end{proof}

\begin{lemma}
 Let $\epsilon \in \left(0,\frac{1}{4}\right)$, $b=\frac{1}{2}+\epsilon$ and $b'=\frac{1}{2}-2\epsilon$. Then for $f \in H^{-b'}(\mathbb{R})$
 \begin{displaymath}
  \norm{\psi_T\displaystyle\int_{0}^{t} f(t')dt'}_{H^{b}}\lesssim T^\epsilon \norm{f}_{H^{-b'}}
 \end{displaymath}
 \label{long lemma}
\end{lemma}
\vspace{-25pt}
\begin{proof}
Firstly,
\begin{align*}
\psi_T\displaystyle\int_{0}^{t}f(t')dt' &= \psi_T\displaystyle\int_{0}^{t}\displaystyle\int_{\mathbb{R}}e^{it\tau}\hat{f}(\tau)d\tau dt'
\\
&=\psi_T\displaystyle\int_{\mathbb{R}}\displaystyle\int_{0}^{t}e^{it\tau}\hat{f}(\tau)dt'd\tau 
\\
&=\psi_T\displaystyle\int_{\mathbb{R}}\hat{f}(\tau)\displaystyle\int_{0}^{t}e^{it\tau}dt'd\tau 
\\
&=\psi_T\displaystyle\int_{\mathbb{R}}\frac{e^{it\tau}-1}{i\tau}\hat{f}(\tau)d\tau
\\
&=\psi_T\displaystyle\int_{\absval{\tau}\leq\frac{1}{T}}\frac{e^{it\tau}-1}{i\tau}\hat{f}(\tau)d\tau-\psi_T\displaystyle\int_{\absval{\tau}\geq\frac{1}{T}}\frac{\hat{f}(\tau)}{i\tau}d\tau+\psi_T\displaystyle\int_{\absval{\tau}\geq\frac{1}{T}}\frac{e^{it\tau}}{i\tau}\hat{f}(\tau)d\tau
\\
&=:\displaystyle\sum_{j=1}^{3} I_j 
\end{align*}
We proceed by estimating each integral $I_j$ for  $j=1,2,3.$ Since
\begin{align*}
\frac{e^{it\tau}-1}{i\tau}=& \frac{e^{it\tau}}{i\tau}-\frac{1}{i\tau}=\sum_{k=0}^{\infty}\frac{\left(it\tau\right)^k}{k!\left(i\tau\right)}-\frac{1}{i\tau}\\
=&\sum_{k=0}^{\infty}\frac{t^k}{k!}\left(i\tau\right)^{k-1}-\frac{1}{i\tau}=\sum_{k=1}^{\infty}\frac{t^k}{k!}\left(i\tau\right)^{k-1},\\
\end{align*}
it follows that
\begin{displaymath}
I_1 =\psi_T\displaystyle\int_{\absval{\tau}\leq\frac{1}{T}}\displaystyle\sum_{k=1}^{\infty}\frac{t^k}{k!}\left(i\tau\right)^{k-1}\hat{f}(\tau)d\tau. 
\end{displaymath}
Accordingly, for $I_1$ we have
\begin{align*}
\norm{I_1}_{H^b} &=\norm{\japbrack{\tau}^b \mathcal{F}_t \left[\displaystyle\int_{\absval{\tau}\leq\frac{1}{T}}\displaystyle\sum_{k=1}^{\infty}\frac{t^k\psi_T}{k!}\left(i\tau\right)^{k-1}\hat{f}(\tau)d\tau \right]}_{L_{\tau}^{2}} 
\\
&=\norm{\japbrack{\tau}^b \displaystyle\int_{\mathbb{R}}e^{-it\tau} \left(\displaystyle\sum_{k=1}^{\infty}\frac{t^k\psi_T}{k!}\displaystyle\int_{\absval{\tau'}\leq\frac{1}{T}}\left(i\tau'\right)^{k-1}\hat{f}(\tau')d\tau' \right)dt}_{L_{\tau}^{2}} 
\\
\end{align*}
Since $\absval{\tau'} \leq T^{-1}$ implies that $\absval{i\tau'}^{k-1} \leq T^{1-k}$, it follows that the above norm is bounded above by
\begin{align*}
&\leq \norm{\japbrack{\tau}^b \displaystyle\int_{\mathbb{R}}e^{-it\tau} \left(\displaystyle\sum_{k=1}^{\infty}\frac{t^k\psi_T}{k!}T^{1-k}\displaystyle\int_{\absval{\tau'}\leq\frac{1}{T}}\hat{f}(\tau')d\tau' \right)dt}_{L_{\tau}^{2}} 
\\
& \leq \norm{\japbrack{\tau}^b \displaystyle\mathcal{F}_t\left[\displaystyle\sum_{k=1}^{\infty}\frac{t^k\psi_T}{k!}T^{1-k}\right]}_{L_{\tau}^{2}} \cdot \displaystyle\int_{\absval{\tau'}\leq\frac{1}{T}}\absval{\hat{f}(\tau')}d\tau'
\\
\end{align*}

\vspace{35pt}

Due to a combination of Minkowski's and H\"older's inequality and the definition of the $H^b$ norm it follows that
\begin{align*}
&\leq \displaystyle\sum_{k=1}^{\infty}\norm{\frac{t^k\psi_T}{k!}T^{1-k}}_{H^{b}}\cdot \displaystyle\int_{\absval{\tau'}\leq\frac{1}{T}}\absval{\hat{f}(\tau')}d\tau'\\
&=\displaystyle\sum_{k=1}^{\infty}\norm{\frac{t^k\psi_T}{k!}}_{H^{b}}\cdot T^{1-k} \displaystyle\int_{\absval{\tau'}\leq\frac{1}{T}}\japbrack{\tau'}^{-b'}\absval{\hat{f}(\tau')}\japbrack{\tau'}^{b'}d\tau'\\
&\leq \displaystyle\sum_{k=1}^{\infty}\norm{\frac{t^k\psi_T}{k!}}_{H^{b}} \cdot \left(\displaystyle\int_{\absval{\tau'}\leq\frac{1}{T}}\japbrack{\tau'}^{2b'}d\tau'\right)^\frac{1}{2} \cdot T^{1-k}\norm{f}_{H^{-b'}} \\
&:=A_1A_2T^{1-k}\norm{f}_{H^{-b'}}
\end{align*}

\vspace{35pt}

It directly follows from lemma \ref{lemma:psiT} that $\norm{t^k \psi_T}_{H^{b}} \leq CT^{k+\frac{1}{2}-b}$, thus
\begin{displaymath}
A_1=\displaystyle\sum_{k=1}^{\infty}\norm{\frac{t^k\psi_T}{k!}}_{H^{b}} \leq \displaystyle\sum_{k=1}^{\infty}\frac{C}{k!}T^{k+\frac{1}{2}-b}.
\end{displaymath}
To bound $A_2$, we make use of the fact that $\japbrack{x}^s := \left( 1+x^2 \right)^\frac{s}{2}$ to obtain
\begin{align*}
A_2 &= \left(\displaystyle\int_{\absval{\tau'}\leq\frac{1}{T}}\japbrack{\tau'}^{2b'}d\tau'\right)^\frac{1}{2}
\\
& \leq \left(\displaystyle\int_{\absval{\tau'}\leq\frac{1}{T}}\left(1+\frac{1}{T^2}\right)^{b'}d\tau'\right)^\frac{1}{2}
\\
& \leq C\left(\displaystyle\int_{\absval{\tau'}\leq\frac{1}{T}}T^{-2b'}d\tau'\right)^\frac{1}{2}
\\
& =CT^{-b'}\left(\displaystyle\int_{\absval{\tau'}\leq\frac{1}{T}}d\tau'\right)^\frac{1}{2}
\\
&=CT^{-b'}\left(\frac{1}{T}+\frac{1}{T} \right)^\frac{1}{2} \\
&=\sqrt{2}CT^{-b'}\left(T^{-1} \right)^\frac{1}{2} \\
&=CT^{-b'-\frac{1}{2}}
\end{align*}

\vspace{35pt}

In conclusion, it transpires that
\begin{align*}
\norm{I_1}_{H^b} &\leq A_1A_2T^{1-k}\norm{f}_{H^{-b'}}
\\
&\leq \left( \displaystyle\sum_{k=1}^{\infty}\frac{C}{k!}T^{k+\frac{1}{2}-b}  \cdot CT^{-b'-\frac{1}{2}} \cdot T^{1-k}\right)\norm{f}_{H^{-b'}}
\\
&\leq C \left(\displaystyle\sum_{k=1}^{\infty} \frac{1}{k!}  T^{k+\frac{1}{2}-b-b'-\frac{1}{2}+1-k}\right)\norm{f}_{H^{-b'}}
\\
&= C\left(\displaystyle\sum_{k=1}^{\infty} \frac{1}{k!}\right)  T^{-b-b'+1}\norm{f}_{H^{-b'}}
\\
&\leq CT^{-b-b'+1}\norm{f}_{H^{-b'}} \\
&=CT^{-\left(\frac{1}{2}+\epsilon\right)-\left(\frac{1}{2}-2\epsilon\right)+1}\norm{f}_{H^{-b'}} \\
&:=C_1T^\epsilon\norm{f}_{H^{-b'}}.
\end{align*}
For $I_2$, we have
\begin{align*}
\norm{I_2}_{H^b}&=\norm{\japbrack{\tau}^b \mathcal{F}_t \left[\psi_T\displaystyle\int_{\absval{\tau}\geq\frac{1}{T}}\frac{\hat{f}(\tau)}{i\tau}d\tau\right]}_{L_{\tau}^{2}} 
\\
&=\norm{\japbrack{\tau}^b \displaystyle\int_{\mathbb{R}}e^{-it\tau}\psi_T\displaystyle\int_{\absval{\tau}\geq\frac{1}{T}}\frac{\hat{f}(\tau)}{i\tau}d\tau dt}_{L_{\tau}^{2}} 
\\
&=\norm{\japbrack{\tau}^b \hat{\psi}_T \displaystyle\int_{\absval{\tau}\geq\frac{1}{T}}\frac{\hat{f}(\tau)}{i\tau}d\tau}_{L_{\tau}^{2}} 
\\
& \leq \norm{\psi_T}_{H^b} \cdot \displaystyle\int_{\absval{\tau}\geq\frac{1}{T}}\frac{\absval{\hat{f}(\tau)}}{\absval{\tau}}\japbrack{\tau}^{-b'}\japbrack{\tau}^{b'}d\tau
\\
& \leq \norm{\psi_T}_{H^b} \cdot \norm{f}_{H^{-b'}} \cdot \left(\displaystyle\int_{\absval{\tau}\geq\frac{1}{T}} \frac{\left(1+\tau^2 \right)^{b'}}{\absval{\tau}^2}d\tau \right)^{\frac{1}{2}}.
\end{align*}

\vspace{35pt}

However, if one lets $\beta=\tau T$, then $d\tau=\frac{1}{T}d\beta$
\begin{align*}
\left(\displaystyle\int_{\absval{\tau}\geq\frac{1}{T}} \frac{\left(1+\tau^2 \right)^{b'}}{\absval{\tau}^2}d\tau \right)^{\frac{1}{2}} &= \left(\displaystyle\int_{\absval{\beta}\geq 1} \frac{T^2}{\beta^2}\left(1+\frac{\beta^2}{T^2}\right)^{b'}\frac{1}{T}d\beta \right)^{\frac{1}{2}}
\\
& \leq C T^{\frac{1}{2}}\left(\displaystyle\int_{\absval{\beta}\geq 1} T^{-2b'}\frac{\beta^{2b'}}{\beta^2}d\beta \right)^{\frac{1}{2}}
\\
&= C T^{\frac{1}{2}-b'}\left(\displaystyle\int_{\absval{\beta}\geq 1} \frac{\beta^{1-4\epsilon}}{\beta^2}d\beta \right)^{\frac{1}{2}}
\\
&= C T^{\frac{1}{2}-b'}\left(\displaystyle\int_{\absval{\beta}\geq 1} \frac{1}{\beta^{1+4\epsilon}}d\beta \right)^{\frac{1}{2}}
\\
&= C T^{\frac{1}{2}-b'}.
\end{align*}
Due to lemma \ref{lemma:psiT} and fact that $\epsilon \in \left(0,\frac{1}{4}\right)$ implies that $\displaystyle\int_{\absval{\beta}\geq 1} \frac{1}{\beta^{1+4\epsilon}}d\beta < \infty$, we have
\begin{align*}
\norm{I_2}_{H^b}&\leq C\norm{\psi_T}_{H^b} \cdot \norm{f}_{H^{-b'}} \cdot T^{\frac{1}{2}-b'} \\
&\leq CT^{\frac{1}{2}-b} \cdot \norm{f}_{H^{-b'}} \cdot T^{\frac{1}{2}-b'}
\\
&=CT^{1-b-b'}\norm{f}_{H^{-b'}}\\
&:=C_2T^{\epsilon}\norm{f}_{H^{-b'}}.\\
\end{align*}
For $I_3$ it follows that
\begin{align*}
\norm{I_3}_{H^b}&=\norm{\japbrack{\tau}^b \mathcal{F}_t \left[\psi_T\displaystyle\int_{\absval{\tau}\geq\frac{1}{T}}\frac{e^{it\tau}}{i\tau}\hat{f}(\tau)d\tau\right]}_{L_{\tau}^{2}} =\norm{\japbrack{\tau}^b \mathcal{F}_t \left[\psi_T \cdot \mathcal{F}^{-1}_\tau \left[\mathbf{1}_{\absval{\tau}\geq\frac{1}{T}}\frac{\hat{f}(\tau)}{i\tau}\right]\left(t\right)\right]}_{L_{\tau}^{2}} \\
&=\norm{\japbrack{\tau}^b \cdot \hat{\psi}_T \ast \left(\mathbf{1}_{\absval{\tau}\geq\frac{1}{T}}\frac{\hat{f}(\tau)}{i\tau}\right)}_{L_{\tau}^{2}}.
\\
\end{align*}
Define $\hat{B}(\tau)=:\mathbf{1}_{\absval{\tau}\geq\frac{1}{T}}\frac{\hat{f}(\tau)}{i\tau}$, then
\begin{align*}
\norm{B}_{H^b}&=\norm{\japbrack{\tau}^b\hat{B}}_{L_{\tau}^{2}}=\norm{\japbrack{\tau}^b\frac{\japbrack{\tau}^{b'}}{\japbrack{\tau}^{b'}}\hat{B}}_{L_{\tau}^{2}}
=\norm{\hat{f}(\tau)\frac{\japbrack{\tau}^{-b'}}{i\tau}\japbrack{\tau}^{b+b'}\mathbf{1}_{\absval{\tau}\geq\frac{1}{T}}}_{L_{\tau}^{2}}
\\
& \leq \norm{f}_{H^{-b'}} \cdot \sup_{\absval{\tau}\geq\frac{1}{T}}\frac{\japbrack{\tau}^{b+b'}}{\absval{\tau}}.
\\
\end{align*}

However,
\begin{displaymath}
\sup_{\absval{\tau}\geq\frac{1}{T}}\frac{\japbrack{\tau}^{b+b'}}{\absval{\tau}} \leq C\sup_{\absval{\tau}\geq\frac{1}{T}}\absval{\tau}^{b'+b-1} =C\sup_{\absval{\tau}\geq\frac{1}{T}}\frac{1}{\absval{\tau}^{1-b-b'}}\leq CT^{1-b-b'} = CT^{\epsilon}. \\
\end{displaymath}

\medskip

As a result,
\begin{displaymath}
\norm{B}_{H^b}\leq \norm{f}_{H^{-b'}}\cdot CT^{1-b-b'}=CT^\epsilon\norm{f}_{H^{-b'}}.  
\end{displaymath}

Hence,
\begin{align*}
&\norm{\japbrack{\tau}^b \cdot \hat{\psi}_T \ast \left(\mathbf{1}_{\absval{\tau}\geq\frac{1}{T}}\frac{\hat{f}(\tau)}{i\tau}\right)}_{L_{\tau}^{2}} = \norm{\japbrack{\tau}^b \cdot \hat{\psi}_T \ast \hat{B}}_{L_{\tau}^{2}}
\\
= &  \left(\displaystyle\int_{\mathbb{R}}\left(\displaystyle\int_{\mathbb{R}}\japbrack{\tau}^b \hat{\psi}_T(\tau-\tau')\hat{B}(\tau') d\tau'\right)^2 d\tau \right)^{\frac{1}{2}}
\\
\leq &\left(\displaystyle\int_{\mathbb{R}}\left(\displaystyle\int_{\mathbb{R}}\left(\japbrack{\tau-\tau'}^b+ \japbrack{\tau'}^b\right)\hat{\psi}_T(\tau-\tau')\hat{B}(\tau') d\tau'\right)^2 d\tau \right)^{\frac{1}{2}}
\\ 
=&\norm{\japbrack{\tau}^b  \hat{\psi}_T \ast \hat{B}+ \hat{\psi}_T \ast \japbrack{\tau}^b\hat{B}}_{L_{\tau}^{2}}\leq \norm{\japbrack{\tau}^b \hat{\psi}_T \ast \hat{B}}_{L_{\tau}^{2}}+\norm{\hat{\psi}_T \ast \japbrack{\tau}^b \hat{B}}_{L_{\tau}^{2}}
\\
 \leq& \norm{\japbrack{\tau}^b \hat{\psi}_T}_{L_{\tau}^{1}}  \norm{\hat{B}}_{L_{\tau}^{2}}+\norm{\hat{\psi}_T}_{L_{\tau}^{1}} \norm{B}_{H^b}
\\
 = & \norm{\japbrack{\tau}^b \hat{\psi}_T}_{L_{\tau}^{1}}  \norm{\mathbf{1}_{\absval{\tau}\geq\frac{1}{T}}\frac{\hat{f}(\tau)}{i\tau}}_{L_{\tau}^{2}}+\norm{\hat{\psi}_T}_{L_{\tau}^{1}} \norm{B}_{H^b}
\\
= & \norm{\japbrack{\tau}^b \hat{\psi}_T}_{L_{\tau}^{1}}  \norm{\japbrack{\tau}^{b'}\japbrack{\tau}^{-b'}\mathbf{1}_{\absval{\tau}\geq\frac{1}{T}}\frac{\hat{f}(\tau)}{i\tau}}_{L_{\tau}^{2}}+\norm{\hat{\psi}_T}_{L_{\tau}^{1}} \norm{B}_{H^b}
\\ 
\leq&  C\sup_{\absval{\tau}\geq\frac{1}{T}}\absval{\tau}^{b+b'-1}\norm{\hat{\psi}_T}_{L_{\tau}^{1}}  \norm{\japbrack{\tau}^{-b'}\hat{f}}_{L_{\tau}^{2}}+\norm{\hat{\psi}_T}_{L_{\tau}^{1}} \norm{B}_{H^b} 
\\
= &C\sup_{\absval{\tau}\geq\frac{1}{T}}\frac{1}{\absval{\tau}^{1-b-b'}}\norm{\hat{\psi}_T}_{L_{\tau}^{1}}  \norm{\japbrack{\tau}^{-b'}\hat{f}}_{L_{\tau}^{2}}+\norm{\hat{\psi}_T}_{L_{\tau}^{1}} \norm{B}_{H^b} 
\\
 \leq & CT^{1-b-b'}\norm{\hat{\psi}_T}_{L_{\tau}^{1}}  \norm{\japbrack{\tau}^{-b'}\hat{f}}_{L_{\tau}^{2}}+\norm{\hat{\psi}_T}_{L_{\tau}^{1}} \norm{B}_{H^b} 
\\
\leq & CT^{1-b'-b}\norm{f}_{H^{-b'}} +CT^{1-b-b'}\norm{f}_{H^{-b'}}
\\
 =&CT^{1-b-b'}\norm{f}_{H^{-b'}}
\\
:=& C_3T^\epsilon\norm{f}_{H^{-b'}}.
\end{align*}
Therefore,
\begin{align*}
\norm{\psi_T\displaystyle\int_{0}^{t}f(t')dt'}_{H^b} &\leq \displaystyle\sum_{j=1}^{3} \norm{I_j}_{H^b} 
\\
& \leq \displaystyle\sum_{i=1}^{3}C_jT^\epsilon\norm{f}_{H^{-b'}}
\\
&=CT^\epsilon\norm{f}_{H^{-b'}}
\end{align*}
and Lemma \ref{long lemma} is proved and we are ready to prove Theorem \ref{main lemma}.
\end{proof}
\begin{proof}
Let $\epsilon \in \left(0,\frac{1}{4}\right)$, $b=\frac{1}{2}+\epsilon$ and $b'=\frac{1}{2}-2\epsilon$. We proceed to prove Theorem \ref{main lemma} from Lemma \ref{long lemma}. Firstly, a straightforward calculation shows that $\norm{u}_{X_i^{b,s}}=\norm{S_i(-t)u}_{H^{b,s}}$. 
\begin{align*}
\norm{S_i(-t)u}_{H^{b,s}} &=\norm{\japbrack{\xi}^2\japbrack{\tau}^b  \japbrack{\absval{\xi}+\absval{\mu}}^s\mathcal{F}_{x,y,t}\left[ S_i(-t)u \right] }_{L_{\xi,\mu,\tau}^{2}}
\\
&=\norm{\japbrack{\xi}^2\japbrack{\tau}^b  \japbrack{\absval{\xi}+\absval{\mu}}^se^{-it\omega_i}\hat{u}\left(\xi,\mu,\tau \right) }_{L_{\xi,\mu,\tau}^{2}}
\\
&=\norm{\japbrack{\xi}^2\japbrack{\tau}^b  \japbrack{\absval{\xi}+\absval{\mu}}^s\hat{u}\left(\xi,\mu,\tau+\omega_i \right) }_{L_{\xi,\mu,\tau}^{2}}
\end{align*}
If we let $\beta=\tau+\omega_i$, then $d\beta=d\tau$ and the above equals
\begin{equation*}
\norm{\japbrack{\xi}^2\japbrack{\beta-\omega_i}^b  \japbrack{\absval{\xi}+\absval{\mu}}^s\hat{u}\left(\xi,\mu,\beta\right)}_{L_{\xi,\mu,\beta}^{2}} =\norm{\japbrack{\xi}^2\japbrack{\tau-\omega_i}^b  \japbrack{\absval{\xi}+\absval{\mu}}^s\hat{u}\left(\xi,\mu,\tau\right) }_{L_{\xi,\mu,\tau}^{2}}
=\norm{u}_{X_i^{b,s}}.
\end{equation*}
Thus, the term
\begin{displaymath}
\norm{\psi_T\displaystyle\int_{0}^{t} S_i(t-t')F_i(w_1(t'),w_2(t'))dt'}_{X_i^{b,s}}
\end{displaymath}
is equal to
\begin{displaymath}
\norm{S_i\left(-t\right)\psi_T\displaystyle\int_{0}^{t} S_i(t-t')F_i(w_1(t'),w_2(t'))dt'}_{H^{b,s}}.
\end{displaymath}
Therefore, 
\begin{align*}
=&\norm{\psi_T\displaystyle\int_{0}^{t} S_i(-t')F_i(w_1(t'),w_2(t'))dt'}_{H^{b,s}}
\\
=&\norm{\japbrack{\xi}^2\japbrack{\tau}^b  \japbrack{\absval{\xi}+\absval{\mu}}^s\mathcal{F}_{x,y,t}\left[\psi_T\displaystyle\int_{0}^{t} S_i(-t')F_i(w_1(t'),w_2(t'))dt'\right] }_{L_{\xi,\mu,\tau}^{2}}
\\
=&\norm{\japbrack{\xi}^2 \japbrack{\absval{\xi}+\absval{\mu}}^s\mathcal{F}_{x,y}\left[\japbrack{\tau}^b \mathcal{F}_{t}\left[\psi_T\displaystyle\int_{0}^{t} S_i(-t')F_i(w_1(t'),w_2(t'))dt'\right]\right]}_{L_{\xi,\mu,\tau}^{2}}
\\
\leq & \norm{\japbrack{\xi}^2 \japbrack{\absval{\xi}+\absval{\mu}}^s\norm{\mathcal{F}_{x,y}\left[\psi_T\displaystyle\int_{0}^{t} S_i(-t')F_i(w_1(t'),w_2(t'))dt'\right]}_{H^{b}}}_{L_{\xi,\mu,\tau}^{2}}
\\
 \leq & CT^{1-b-b'}\norm{\japbrack{\xi}^2 \japbrack{\absval{\xi}+\absval{\mu}}^s\norm{\mathcal{F}_{x,y}\left[S_i(-t')F_i(w_1(t'),w_2(t'))\right]}_{H^{-b'}}}_{L_{\xi,\mu,\tau}^{2}}
\\
= & CT^{1-b-b'}\norm{S_i(-t')F_i(w_1(t'),w_2(t'))}_{H^{-b',s}}
\\
= & CT^{1-b-b'}\norm{F_i(w_1(t'),w_2(t'))}_{X_i^{-b',s}}
\\
= & CT^{\epsilon}\norm{F_i(w_1(t'),w_2(t'))}_{X_i^{-b',s}}
\end{align*}
\end{proof}
\section*{Acknowledgements} 
 \small{J.A. is grateful for the support from the College of Computing, Artificial Intelligence, Robotics, and Data Science at Saint Leo University.}
\section*{Appendix}
\label{sec:appendix}
The purpose of this appendix is to cover the derivation of a recursive formula for the terms in the Taylor expansion of the Dirichlet-Neumann operator. To accomplish this we follow the work of Craig et al. \cite{CG}. The method utilized hinges on the following theorem, for a proof we refer the reader to \cite{CHM}.
\begin{theorem}
There exists a constant $C=O\left(h_0\right)$ such that the operator $G(\eta)$ is analytic for $\eta \in C^\infty$ in a ball of radius $C$ about the point $\eta=0$.
\end{theorem}
As a direct result, the Taylor series expansion of $G(\eta)$ about $0$ takes the form
\begin{equation*}
G(\eta) = \sum_{j \geq 0} G_j(\eta),
\end{equation*}
for spatial dimensions $n=2,3$, where each $G_j(\eta)$ is homogeneous of degree $j$ in $n$. We then proceed to systematically compute the terms $G_j(\eta)$ via a recursive formula introduced in \cite{CSS}. In what follows we define the fluid domain to be the region 
$$S(\eta)=\{(x,y) :x \in S_0 , -h_0 \leq y \leq \eta(x)\}.$$

One should observe that $S(\eta)$ is bounded below by the rigid bottom $y=-h_0$ and above by the graph
$S_g=\{(x,y) :x \in S_0,y=\eta(x)\}$, where $S_0=\{(x,0,z) :(x,z) \in \mathbb{R}^2 \}$ is the undisturbed free surface. The elliptic boundary value problem to be solved is
\begin{equation*}
\begin{cases}
\Delta \phi = 0 \quad \text{in} \quad S(\eta) &\\
\phi = \Phi \quad \text{on} \quad S_G &\\
\frac{\partial \phi}{\partial n}=0 \quad \text{on} \quad y=-h_0 &\\
\end{cases}
\label{BVP}
\end{equation*}
coupled with the appropriate periodic or asymptotic conditions on $\phi$. It is well known that there exists a family of solutions to the problem for $k \in \mathbb{R}^{n-1}$. In accordance with the above discussion, we define
\begin{equation}
\phi_k(x,y)=e^{i(x \cdot k)}\cosh(\absval{k}(y+h_0)).
\label{BVP sol}
\end{equation}
These functions have the boundary values $\Phi_k(x,y)=e^{i(x \cdot k)}\cosh(\absval{k}(\eta(x)+h_0))$, and are compatible with the asymptotic conditions as $\absval{x} \rightarrow \infty$, or with conditions of periodicity for appropriately chosen wavevectors $k$. Now since the family of functions $\phi_k$ satisfy the equation
\begin{equation}
G(\eta)\Phi_k=\left[ \frac{\partial \phi_k}{\partial y}- \frac{\partial \eta}{\partial x} \cdot \frac{\partial \phi_k}{\partial x} \right]_{y=\eta}.
\label{DN expression}
\end{equation}
Now, we substitute (\ref{BVP sol}) into (\ref{DN expression}) and expand the Dirichlet-Neumann operator and hyperbolic functions to obtain
\begin{align*}
& \sum_{j \ \text{even}} \frac{\left(\absval{k} \eta \right)^j}{j!} \left(\sinh(\absval{k}h_0) - i \frac{\partial \eta}{\partial x} \cdot k  \cosh(\absval{k}h_0)   \right)e^{i(x \cdot k)} \\
&+ \sum_{j \ \text{odd}}  \frac{\left(\absval{k} \eta \right)^j}{j!} \left(\cosh(\absval{k}h_0) - i \frac{\partial \eta}{\partial x} \cdot k  \sinh(\absval{k}h_0)   \right)e^{i(x \cdot k)} \\
=& \left(\sum_{m \geq 0} G_m(\eta)\right)\left(\sum_{j \ \text{even}} \frac{\left(\absval{k} \eta \right)^j}{j!} \cosh(\absval{k}h_0) e^{i(x \cdot k)} +\sum_{j \ \text{odd}} \frac{\left(\absval{k} \eta \right)^j}{j!} \sinh(\absval{k}h_0) e^{i(x \cdot k)} \right).
\end{align*} 
We then proceed to obtain the recursive formula by equating terms of the same degree in $\eta$. For $j=0$ it follows that
\begin{equation*}
G_0 e^{i(x \cdot k)}=\absval{k}\tanh(\absval{k}h_0)e^{i(x \cdot k)}.
\label{DchNeum}
\end{equation*}
By utilizing Fourier analytic techniques, one obtains a forumula for $G_j \Phi$ , valid for sufficient $\Phi$ and $j=0,1,2$. 
\begin{align*}
G_0(\eta) \Phi \left(x\right) &=\absval{D} \tanh(\absval{D} h_0)\Phi \left(x\right),
\end{align*}
For convenience, we suppress the $\Phi$ dependency and write
\begin{align*}
 G_0 (\eta) =&\absval{D} \tanh(\absval{D} h_0),  \\
 G_1 (\eta) =&\absval{D} \eta \absval{D} - \absval{D} \tanh(\absval{D} h_0)\eta \absval{D} \tanh(\absval{D} h_0), \nonumber \\
 G_2 (\eta) =&-\frac{1}{2}\Big[ \absval{D}^2 \eta^2 \absval{D} \tanh(\absval{D} h_0)\nonumber \\
&+ \absval{D} \tanh(\absval{D} h_0) \eta^2 \absval{D}^2 \nonumber\\
&- 2\absval{D} \tanh(\absval{D} h_0)\eta \absval{D} \tanh(\absval{D} h_0)\eta \absval{D} \tanh(\absval{D} h_0)\Big]. \nonumber\\
\end{align*}

More generally, we obtain
{\small
\begin{equation*}
  G_j(\eta) =  \left\{
\begin{array}{rl}
\displaystyle\frac{1}{j!}\Big[\eta^j \absval{D}^{j+1}\tanh(\absval{D}h_0) - i \frac{\partial}{\partial x} \left(\eta^j\right) \cdot D \absval{D}^{j-1}\tanh(\absval{D}h_0)\Big]  & \\
\displaystyle-\sum_{\nu < j, \nu \ \text{even}} \frac{1}{(j-\nu)!} G_\nu(\eta)\eta^{j-\nu} \absval{D}^{j-\nu} & \text{if $j$ is even} \\
-\displaystyle\sum_{\nu < j, \nu \ \text{odd}} \frac{1}{(j-\nu)!} G_\nu(\eta)\eta^{j-\nu} \absval{D}^{j-\nu}\tanh(\absval{D}h_0)  &  \\
& \\
\displaystyle\frac{1}{j!}\Big[\eta^j \absval{D}^{j+1} - i \frac{\partial}{\partial x} \left(\eta^j\right) \cdot D \absval{D}^{j-1}\Big] & \\
\displaystyle-\sum_{\nu < j, \nu \ \text{even}} \frac{1}{(j-\nu)!} G_\nu(\eta)\eta^{j-\nu} \absval{D}^{j-\nu}\tanh(\absval{D}h_0) &\text{if $j$ is odd} \\
\displaystyle-\sum_{\nu < j, \nu \ \text{odd}} \frac{1}{(j-\nu)!} G_\nu(\eta)\eta^{j-\nu} \absval{D}^{j-\nu} &\\
    \end{array} \right.
\label{DchNeumGEN}
\end{equation*}
}
The above equations constitute the recursive formula for the terms $G_j(\eta)$, as shown in \cite{CG}.
$\,$

$\,$

\bibliographystyle{siam}
\bibliography{NONKPbio} 

\section*{Affiliations} 
\noindent \small{College of Computing, Artificial Intelligence, Robotics, and Data Science, Saint Leo University, Saint Leo, FL 33574, USA}
\end{document}